\documentclass[11pt]{amsart}

\usepackage[T1]{fontenc}
\usepackage[utf8]{inputenc}
\usepackage{lmodern}
\usepackage{stmaryrd}

\usepackage{amssymb}

\newtheorem{theorem}{Theorem}[section]
\newtheorem{lemma}[theorem]{Lemma}
\newtheorem{proposition}[theorem]{Proposition}

\newtheorem{corollary}[theorem]{Corollary}

\theoremstyle{definition}
\newtheorem{remark}[theorem]{Remark}
\newtheorem{question}[theorem]{Question}
\newtheorem{definition}[theorem]{Definition}
\newtheorem{problem}[theorem]{Problem}

\usepackage[american]{babel}
\usepackage[babel, final]{microtype}
\usepackage[all]{xy}
\usepackage{ifpdf}
\usepackage{comment}
\usepackage{multirow}
\usepackage{pinlabel}
\usepackage{geometry}
\usepackage[pdftex, dvipsnames]{xcolor}
\usepackage{amsmath,amsthm,amssymb,amsfonts,mathrsfs,tikz,tikz-cd,bm,verbatim,mathtools,hyperref,caption,enumitem,csquotes,float}

\hypersetup{
    colorlinks=true,
    linkcolor=NavyBlue,
    citecolor=NavyBlue
}

\newcommand{\cH}{\mathcal{H}}

\newcommand{\ZZ}{\mathbb{Z}}

\newcommand{\CC}{\mathbb{C}}
\newcommand{\RR}{\mathbb{R}}
\newcommand{\cF}{\mathcal{F}}
\newcommand{\CP}{\mathbb{CP}}
\newcommand{\cC}{\mathcal{C}}
\newcommand{\cT}{\mathcal{T}}
\newcommand{\cS}{\mathcal{S}}
\newcommand{\cK}{\mathcal{K}}
\newcommand{\DD}{\mathbb{D}}
\newcommand{\betat}{\widetilde{\beta}}

\newcommand{\Hcheck}{\check{H}}
\newcommand{\del}{\partial}

\newcommand{\alphas}{\mathbf{\alpha}}
\newcommand{\betas}{\mathbf{\beta}}
\newcommand{\gammas}{\mathbf{\gamma}}

\title{Symplectic forms on trisected 4-manifolds}
\date{March 2022}

\begin{document}

\author{Peter Lambert-Cole}
\address{Department of Mathematics, University of Georgia, Athens, GA 30602}
\email{\href{plc@uga.edu}{plc@uga.edu}}

\begin{abstract}
Previously work of the author with Meier and Starkston showed that every closed symplectic manifold $(X,\omega)$ with a rational symplectic form admits a trisection compatible with the symplectic topology.  In this paper, we describe the converse direction and give explicit criteria on a trisection of a closed, smooth 4-manifold $X$ that allows one to construct a symplectic structure on $X$. Combined, these give a new characterization of 4-manifolds that admit symplectic structures.  This construction motivates several problems on taut foliations, the Thurston norm and contact geometry in 3-dimensions by connecting them to questions about the existence, classification and uniqueness of symplectic structures on 4-manifolds.
\end{abstract}

\maketitle

\section{Introduction}

Let $X$ be a smooth, closed, oriented 4-manifold. It is of great interest in low-dimensional topology to determine whether $X$ admits a symplectic structure and if so, to classify all such symplectic structures.  For example, work of Donaldson and Gompf showed that $X$ admits a symplectic structure if and only if it admits a compatible Lefschetz pencil.  This essentially reduces the problem to
algebra and mapping class groups of surfaces.  

Suppose that $\cT$ is a {\it trisection} of $X$, which is a particular decomposition into three 4-dimensional 1-handlebodies.

\begin{definition}
Let $X$ be a closed, smooth,oriented 4-manifold.  A {\it trisection} of $X$ is a decomposition $X = Z_1 \cup Z_2 \cup Z_3$ such that
\begin{enumerate}
    \item each sector $Z_{\lambda}$ is a 4-dimensional 1-handlebody,
    \item each double intersection $H_{\lambda} = Z_{\lambda} \cap Z_{\lambda+1}$ is a 3-dimensional 1-handlebody of genus $g$, and
    \item the triple intersection $\Sigma = Z_1 \cap Z_2 \cap Z_3$ is a closed, orientable surface of genus $g$.
\end{enumerate}
\end{definition}

The trisection $\cT$ and the 4-manifold $X$ itself can be encoded (up to diffeomorphism) by a {\it trisection diagram} $(\Sigma;\alphas,\betas,\gammas)$, which consists of a closed, oriented surface and three cut systems of curves.  This data determines how to build a trisected 4-manfiold from the `inside-out'.  The cut systems determine how to attach each double intersection $H_{\lambda}$ to the central surface to produce the {\it spine} $H_1 \cup H_2 \cup H_3$.  Capping off by attaching the 4-dimensional sectors $Z_1,Z_2,Z_3$ is unique up to diffeomorphism by the well-known result of Laudenbach and Poenaru \cite{LP}.

The following is a natural question to ask.

\begin{question}
Is there a criterion on a trisection that ensures $X$ admits a symplectic structure?
\end{question}

However, answering this question has proven surprising difficult.  The first problem is to find the correct notion of compatibility between a symplectic structure and a trisected 4-manifold.  For example, the central surface is nullhomologous, therefore it cannot be symplectic, and in general has negative Euler characteristic, therefore it cannot be Lagrangian.  The notions of {\it Weinstein trisections} of symplectic 4-manifolds and {\it Stein trisections} of complex manifolds were introduced in \cite{LM-Rational}.  Some recent progress has been made in constructing Stein trisections \cite{LM-Rational,LC-Balls,Zupan-Projective}.  Moreover, by using Auroux's work on representing symplectic 4-manifolds as branched coverings of $\CP^2$ \cite{Auroux}, it was shown that every symplectic 4-manifold admits a Weinstein trisection \cite{LMS}.  This construction is evidence that the latter structure is the conceptually correct one to look for.

A second difficulty is that the compatibility between trisections and symplectic topology does not appear to fit into existing conceptual frameworks.  For example, the boundary of a relative trisection inherits an open book decomposition \cite[Theorem 20]{GK-Tri}.  This suggests that for a compact, symplectic manifold with contact boundary, the open book induced by the relative trisection should support the contact structure.  However, consider the complement of a complex curve $\cC_d$ in $\CP^2$.  The contact structure on the boundary of $\CP^2 \setminus \nu(\cC_d)$ is Stein-fillable, hence tight.  
Naively applying the method of Miller-Kim \cite{Miller-Kim} to the curves in bridge position given in \cite{LM-Rational} can yield an induced open book that is not right-veering.  In particular, this open book supports a contact structure that is {\it not} tight.

In this paper, we describe how to construct a symplectic structure from geometric data on the spine of a trisected 4-manifold.  

\subsection{Representing cohomology via cocycles}

The construction of the closed 2-form $\omega$ on the trisected 4-manifold in Theorem \ref{thrm:main} inherently determines the class $[\omega] \in H^2_{DR}(X)$.  Before stating the theorem, we review how to interpret cohomology classes on trisected 4-manifolds.

Feller-Klug-Schirmer-Zemke \cite{FKSZ} and Florens-Moussard \cite{FM} gave presentations of the singular homology groups of a 4-manifold in terms of a trisection diagram.  This was reinterpreted in \cite{LC-Cohomology} as the Cech cohomology groups of certain presheaves with respect to an open cover determined by the trisection.  

In particular, let $\cH^i_{DR}$ denote the presheaf on $X$ that assigns to each open set $U$ the group $H^i_{DR}(U)$.  By abuse of notation, let $\cT = \{U_1,U_2,U_3\}$ denote the open cover of $X$ by three sets obtained by slightly thickening the three sectors of the trisection.  In particular, $U_{\lambda}$ is an open neighborhood of $X_{\lambda}$; $U_{\lambda} \cap U_{\lambda+1}$ is an open neighborhood of the handlebody $H_{\lambda}$, and $U_1 \cap U_2 \cap U_3$ is an open neighborhood of the central surface $\Sigma$.

By analyzing the Cech-to-derived-functor spectral sequence, applied to the open cover $\cT$ and presheaf $\cH^i_{DR}$, one can show there is an isomorphism (\cite{LC-Cohomology})
\[H^2_{DR}(X) \cong \Hcheck^1(\cT,\cH^1_{DR})\]
In words, this states that a class $[\omega] \in H^2_{DR}(X)$ can be represented by a 1-cocycle of closed 1-forms.  Specifically, there is a triple $(\beta_1,\beta_2,\beta_3)$, where $\beta_{\lambda}$ is a closed 1-form on $H_{\lambda}$, and
\[\beta_1 + \beta_2 + \beta_3 = 0\]
when restricted to the central surface $\Sigma$. \\

\subsection{Main theorem}

\begin{definition}
    Let $X$ be a closed, oriented 4-manifold and let $\cT$ be a trisection of $X$ with central surface $\Sigma$ and 3-dimensional sectors $H_1,H_2,H_3$.  A {\it symplectic 1-cocycle} for $(X,\cT)$ is a triple $(\betat_1,\betat_2,\betat_3)$ of closed, nonvanishing 1-forms on the three handlebodies such that the restrictions $\{\beta_{\lambda} = \betat_{\lambda|_{\Sigma}} \}$ satisfy the following conditions 
    \begin{enumerate}
        \item the cocycle condition:
        \[ \beta_1 + \beta_2 + \beta_3 = 0\]
        \item the pairwise positivity condition
        \[\beta_{\lambda} \wedge \beta_{\lambda+1} \geq 0\]
        \item they all vanish at the same $2g-2$ points $b_1,\dots,b_{2g-2}$ in $\Sigma$ and there exists a smooth function $f: \Sigma \rightarrow \RR$ and neighborhoods $\{b_i \in U_i\}$ such that $f|_{U_i} = \pm 1$ and $\sum_i f(b_i) = 0$ and
        \[\betat_{\lambda} = \beta_{\lambda} + d(sf)\]
        on a collar neighborhood of $\partial H_{\lambda} = \Sigma$.
    \end{enumerate}
\end{definition}

\begin{theorem}
\label{thrm:main}
Let $X$ be a closed, oriented 4-manifold and let $\cT$ be a trisection of $X$ with central surface $\Sigma$.  Let $(\betat_1,\betat_2,\betat_3)$ be a symplectic 1-cocycle for $(X,\cT)$.  Then
\begin{enumerate}
    \item there is a symplectic form $\omega$ on $\nu(H_1 \cup H_2 \cup H_3)$ representing the class
\[(\betat_1,\betat_2,\betat_3) \in \Hcheck^1(\cT,\cH^1_{DR}) \cong H^2_{DR}(X) \cong H^2_{DR}(\nu(H_1 \cup H_2 \cup H_3))\]
In addition, there exists inward-pointing Liouville vector fields for $\omega$ along each of the three boundary components $\widehat{Y}_1, \widehat{Y}_2, \widehat{Y}_3$.
\item If each of the induced contact structures $(\widehat{Y}_{\lambda},\widehat{\xi}_{\lambda})$ are tight, then there exists a global symplectic form $\omega$ on $X$ representing the class \[(\beta_1,\beta_2,\beta_3) \in \Hcheck^1(\cT,\cH^1_{DR}) \cong H^2_{DR}(X)\]
that is well-defined up to diffeomorphsim.

\item an $n$-parameter family of symplectic 1-cocycles extends to an $n$-parameter family of symplectic forms $\omega$.  In particular, if the class $(\betat_1,\betat_2,\betat_3) \in \Hcheck^1(\cT,\cH^1_{DR}) \cong H^2_{DR}(X)$ remains constant in the family, then the constructed symplectic forms are isotopic.
\end{enumerate}
\end{theorem}

\subsection{How can the construction fail?}

Many smooth 4-manifolds are known to not admit symplectic structures.  Where might the obstructions lie when considered in light of Theorem \ref{thrm:main}?

\begin{enumerate}
    \item {\bf No triple $\{\betat_{\lambda}\}$ of {\it closed, nonvanishing} 1-forms exist}.  This may be for cohomological reasons, as $[\beta_{\lambda}]$ is not in the image of the restriction map $H^1_{DR}(H_{\lambda}) \rightarrow H^1_{DR}(\Sigma)$.  The 1-cocycle $(\beta_1,\beta_2,\beta_3)$ encodes a class $[\omega]$ in $H^2_{DR}(X)$ and $[\beta_1] \wedge [\beta_2]$ is positive in $H^2_{DR}(\Sigma)$ if and only if $[\omega] \wedge [\omega]$ is positive in $H^4_{DR}(X)$.  In particular, if $b_2^+(X) = 0$, then the required cocycle cannot exists.  Secondly, if there is no cohomological obstruction, any triple $\{\betat_{\lambda}\}$ of closed 1-forms may be forced to vanish in some points.
    \item {\bf The induced contact structure $\widehat{\xi}_{\lambda}$ is not tight}.  This may happen at the level of homotopy.  The contact structure is a 2-plane field, which are classified up to homotopy by their Euler class in $H^2(Y)$ and a 3-dimensional obstruction.  In addition, the contact structure is homotopic to the field of $J$-complex tangencies along $Y$, where $J$ is an $\omega$-tame almost-complex structure.  This extends to an almost-complex structure on the 4-dimensional sector if and only if $\xi$ has vanishing Euler class and is homotopic to the unique tight contact structure.  Consequently, if $X$ does not admit an almost-complex structure, then at least one of the constructed contact structures must be overtwisted.  Secondly, in the unique correct homotopy class of 2-plane fields, there exists one tight and one overtwisted contact structure.  Therefore, the construction may simply yield the overtwisted one.
\end{enumerate}

The 4-manifold $3 \CP^2$ has $b_2^+ > 0$ and admits an almost-complex structure, but is known not to admit a symplectic structure.  Moreover, it admits a genus 3 trisection.  It would be interesting to understand how no possible choice of data on this trisection can satisfy the conditions of Theorem \ref{thrm:main}.  Consequently, this motivates the task of classifying the grafted contact structures that can be constructed on the genus 3 Heegaard splitting of $S^3$.

\subsection{Symplectic structures and trisection diagrams}

A trisected 4-manifold can be encoded by a {\it trisection diagram}, which is a tuple $(\Sigma; \alphas,\betas,\gammas)$ consisting of the central surface $\Sigma$ and three cut systems of curves $\alphas,\betas,\gammas$ that determine how the handlebodies $H_1,H_2,H_3$ are glued in.  There has been significant interest in finding a corresponding diagrammatic theory of symplectic trisections.  One approach is a so-called {\it multisection with divides}, introduced by Islambouli and Starkston \cite{IS-Divides}.

The restrictions $\{ \beta_{\lambda} = \betat_{\lambda}|_{\Sigma} \}$ determine a different picture on the central surface, namely three transverse singular foliations $\cF_1,\cF_2,\cF_3$.  Moreover, if we suppose that $\beta_1,\beta_2,\beta_3$ are {\it integral} 1-forms, then they define maps
\[\beta_1,\beta_2,\beta_3: \Sigma \longrightarrow S^1\]
with $2g - 2$ nondegenerate critical points.  The preimage of a regular value is a multicurve and for a fixed $\beta_{\lambda}$, the union of these multicurves determine a pants decomposition of $\Sigma$ -- i.e. a maximal system of $3g -3$ curves on $\Sigma$.  

The foliations $\cF_{\lambda}$ are not transverse to the boundary $\del H_{\lambda} = \Sigma$, but they do satisfy a condition analogous to being taut, namely the existence of a closed transversal intersecting every leaf.  The leaves of a taut foliation are known to minimize the Thurston norm \cite{Thurston}.  If $\betat_{\lambda}$ is integral, then compact leaves of $\cF_{\lambda}$ also minimize the generalized Thurston norm (specifically, Scharlemann's generalization \cite{Scharlemann}).

This begs the question of whether the extensions $\{\betat_{\lambda}\}$ are necessary pieces of data or are uniquely determined by data defined in terms of the central surface $\Sigma$.

\begin{question}
Given the data of
\begin{enumerate}
    \item an intrinsically harmonic, integral 1-form $\beta$ on $\Sigma$ with nondegenerate zero set $B$,
    \item a Lagrangian subspace $L \subset H^1(\Sigma,\ZZ)$ containing $[\beta]$, and
    \item a class $c_1 \in H_1(\Sigma,B;\ZZ)$
\end{enumerate}
Is there (a unique) handlebody $H$ with $\partial H = \Sigma$ such that
\begin{enumerate}
    \item $L = \text{image}(H^1(H) \rightarrow H^1(\Sigma))$, and
    \item for each regular value $\theta$ for $\beta: \Sigma \rightarrow S^1$, the Thurston norm of the unique element $[F] \in H_2(H,\beta^{-1}(\theta))$ is given by the pairing $\langle c_1, \beta^{-1}(\theta) \rangle$?
\end{enumerate}
\end{question}

\subsection{Grafted contact structures}

In the course of the proof of Theorem \ref{thrm:main}, three contact structures are constructed in Theorem \ref{thrm:main}.  In order to cap off the 4-dimensional sectors with strong symplectic fillings, it is necessary that these contact structures be tight.

Intuitively, these contact structures are obtained from `taut' foliations by a so-called {\it grafting} process.  This is a very natural operation that has nonetheless been hardly studied.  The inspiration arises from complex analysis and geometry.  Let $X$ be a complex surface.  A subset $P \subset X$ is an {\it analytic polyhedron} if there exists an open set $U$ containing the closure of $P$ and holomorphic functions $f_1,\dots,f_n$ on $U$ such that
\[P = \{z \in U : |f_1(z)|,\dots,|f_n(z)| \leq 1\}\]
For a generic choice of $\{f_1,\dots,f_n\}$, the boundary of $P$ is a smooth manifold with corners.

The top-dimensional faces of $\partial P$ are Levi-flat, which means they are foliated by complex codimension 1 holomorphic submanifolds.  The function 
\[g = \text{max}(|f_1|,\dots,|f_n|)\]
is a continuous, plurisubharmonic function on $U$ and $P$ is the sublevel set $ g^{-1}([0,1])$  If $X$ is a Stein surface, then $g$ can be approximated by a smooth, strictly plurisubharmonic function.$\widetilde{g}$.  The sublevel set $\widetilde{P} = \widetilde{g}^{-1}([0,1])$ is a compact Stein domain and the field of complex tangencies along $\partial \widetilde{P} = \widetilde{g}^{-1}(1)$ is a contact structure.

Given the importance of tightness in Theorem \ref{thrm:main}, it would be very interesting to produce criteria that ensures the resulting grafted contact structure is tight.

\begin{problem}
Give topological criteria so that the contact structures obtained in Theorem \ref{thrm:main} are tight.
\end{problem}

Establishing overtwistedness criteria may prove easier, as there may be an analogous criteria to {\it non-right-veering monodromy} for an open book decomposition.

\subsection{Symplectic cone}

Let $X$ be a closed, oriented 4-manifold.  The {\it symplectic cone} of $X$ is the subset of $H^2_{DR}(X)$ that can be represented by a symplectic form.  It is a cone since any nonzero multiple of a symplectic form is also a symplectic form.  Given a trisection $\cT$, we define the {\it symplectic cone of $\cT$} to be the subset of $H^2_{DR}(X)$ obtained by constructing symplectic forms via Theorem \ref{thrm:main} from symplectic 1-cocycles.  

\begin{question}
    Is the symplectic cone of $(X,\cT)$ always equal to the symplectic cone of $X$?
\end{question}

We can obtain partial information about the symplectic cone of a trisection.  Suppose that $(X,\omega)$ is constructed via Theorem \ref{thrm:main} from a symplectic 1-cocycle $\{\betat_{\lambda}\}$ and each $\betat_{\lambda}$ is a rational closed 1-form.  Then by integrating, we can identify $\betat_{\lambda}$ with a map
\[\betat_{\lambda}: H_{\lambda} \rightarrow S^1\]
Consequently, all regular leaves are compact.  Let $S_{\lambda} \subset H_2(H_{\lambda},\Sigma;\ZZ) \cong H^1(H_{\lambda};\ZZ)$ denote the positive cone generated by these compact leaves.  Let $\cS(\cT,\betat)$ be the set of triples $(F_1,F_2,F_3) \in S_{1} \oplus S_2 \oplus S_3$ that satisfy the cocycle condition
\[F_1 + F_2 + F_3 = 0\]
when restricted to $\Sigma$.  The set $\cS(\cT,\betat)$ generalizes the cone of curves on a complex surface.

\begin{proposition}
    Let $\{\betat_{\lambda}\}$ be a rational symplectic 1-cocycle defining on a global symplectic structure on $(X,\cT)$.  Then for every $(F_1,F_2,F_3) \in \cS(\cT,\betat)$, there is a symplectic 1-cocycle
    \[\{\betat_{\lambda} + PD(F_{\lambda}) \}\]
    that defines a global symplectic form on $X$.  In particular, $\{\betat_{\lambda} + PD(F_{\lambda})\}$ lies in the symplectic cone of $(X,\cT)$.
\end{proposition}

\begin{proof}
    Let $F_{\lambda}$ be a finite, real-linear combination of surfaces in $S_{\lambda}$.  Then we can construct Poincare duals to the compact leaves in such a way that each $\betat_{\lambda} + PD(F_{\lambda})$ is a closed, nonvanishing 1-form on $H_{\lambda}$.  Since the triples $(\betat_1,\betat_2,\betat_3)$ and $(F_1,F_2,F_3)$ satisfy the cocycle condition, so do their sum.  Since $F_{\lambda}$ is a positive linear combination of leaves of $\text{ker}(\betat_{\lambda})$, the inequality $\beta_{\lambda} \wedge \beta_{\lambda+1} \geq 0$ implies that we can assume
    \[ \beta_{\lambda} \wedge PD(\partial F_{\lambda+1}), PD(\partial F_{\lambda}) \wedge \beta_{\lambda+1}, PD(\partial F_{\lambda}) \wedge PD(\partial F_{\lambda+1}) \geq 0\]
    Consequently,
    \[\{ \betat_{\lambda} + t PD(F_{\lambda}) \}\]
    for $t \in [0,1]$ is a family of symplectic 1-cocycles that determine a 1-parameter homotopy of global symplectic structures on $(X,\cT)$.
\end{proof}

\subsection{Classification of symplectic structures}

There are several natural equivalence relations on symplectic structures on a 4-manifold $X$:
\begin{enumerate}
    \item the pair $\omega_1,\omega_2$ are {\bf isotopic} if they are connected by a path of cohomologous symplectic forms,
    \item the pair $\omega_1,\omega_2$ are {\bf homotopic} if they are connected by a path of symplectic forms,
    \item the pair $\omega_1,\omega_2$ are {\bf diffeomorphic} if there exists a diffeomorphism $\phi: X \rightarrow X$ such that $\phi^*(\omega_2) = \omega_1$, and 
    \item the pair $\omega_1,\omega_2$ are {\bf deformation-equivalent} if there exists a diffeomorphism $\phi: X \rightarrow X$ such that $\omega_1$ and $\phi^*(\omega_2)$ are homotopic.
\end{enumerate}
We refer the reader to the excellent survey of Salamon on the classification and uniqueness of symplectic forms on 4-manifolds \cite{Salamon}, which has also been incorporated into the most recent edition of \cite{McDuffSalamon1}.

\begin{question}
Suppose that $[\omega_1] = [\omega_2]$.  Are they diffeomorphic?
\end{question}

The construction here gives one potential approach to this question: The symplectic structure obtained (a) encodes the cohomology class of $[\omega]$, (b) is well-defined by the data up to diffeomorphism, and (c) extends to 1-parameter families of the data.  Hence studying the moduli of data may be profitable in classifying symplectic structures.  Moreover, this construction means that different symplectic forms can be directly compared.  

\begin{theorem}[\cite{LMS}]
Let $\omega_1,\omega_2$ be rational symplectic forms on $X$.  Then there exists a single smooth trisection $\cT$ of $X$ such that $\omega_1,\omega_2$ can both be constructed via Theorem \ref{thrm:main} on $\cT$.
\end{theorem}

\begin{question}
    Given a pair of cohomologous symplectic 1-cocycles $\{\betat_{\lambda}\}$ and $\{\betat'_{\lambda}\}$ on $(X,\cT)$, are they connected by a 1-parameter family of symplectic 1-cocycles?
\end{question}

The condition that $\beta_{\lambda} \wedge \beta_{\lambda+1} \geq 0$, with $2g-2$ nondegenerate zeros, implies that we can choose a conformal class of metrics $\{\lambda^2 g\}$ on $\Sigma$ such that $\beta_{\lambda},\beta_{\lambda+1}$ are harmonic with respect to the Riemannian metrics in this conformal class and furthermore $\ast_{\lambda^2 g} \beta_{\lambda} = \beta_{\lambda+1}$.  Equivalently, we can encode this data in terms of a pair $(j,\zeta)$ where $j$ is a complex structure on $\Sigma$ and $\zeta$ is a holomorphic 1-form, and recover $\beta_{\lambda},\beta_{\lambda+1}$ as the real and imaginary parts of $\zeta$.

In particular, consider the space $\Omega \mathcal{M}_g$ of holomorphic 1-forms on Riemann surfaces of genus $g$.  A point $(C,\zeta)$ in $\Omega \mathcal{M}_g$ is one necessary piece of data in Theorem \ref{thrm:main}.  Isoperiodic deformations of $(C,\zeta)$, which are deformations that preserve $[\zeta] \in H^1_{DR}(\Sigma,\CC)$, then lead to a family of symplectic forms in the same DeRham cohomology class in $H^2_{DR}(X)$.  The connectivity and dynamics of leaves of the isoperiodic foliation on $\Omega \mathcal{M}_g$ was determined by Calsamiglia, Deroin and Francaviglia \cite{CDF-Transfer}.

\section{Inspiration}

We describe the inspiration for Theorem \ref{thrm:main} by working out the geometry of $(\CP^2,\omega_{FS})$ with respect to the standard trisection and then describing how the geometry pulls back to an arbitrary projective surface via a holomorphic branched covering.

\subsection{The Fubini-Study form on $\CP^2$}

The construction for a general trisection is analogous to the construction of the Fubini-Study form on $\CP^2$.  

In particular, we have that
\[\omega_{FS} = -\frac{i}{2} \del \overline{\del} \log (|z_1|^2 + |z_2|^2 + |z_3|^2)\]
The Ka\"hler potential $\phi = \log (|z_1|^2 + |z_2|^2 + |z_3|^2)$ is well-defined on homogeneous coordinates, up to adding a constant, hence its exterior derivatives are well-defined.  In the chart $z_{\lambda-1} = 1$, we can equivalently describe the form by
\[\omega_{FS} = d d^{\CC} \log ( 1 + |z_{\lambda}|^2 + |z_{\lambda+1}|^2)\]
where $d^{\CC}f := df(J-)$ for a function $f$ and $J: T\CP^2 \rightarrow \CP^2$ the almost-complex structure on the tangent bundle.  In particular, the form
\[\alpha_{\lambda} = d^{\CC} \log ( 1 + |z_{\lambda}|^2 + |z_{\lambda+1}|^2)\]
is a primitive for $\omega_{FS}$ in this coordinate chart.  In polar coordinates $z_j = r_j e^{i \theta_j}$, we have that
\begin{align*}
    \alpha_{\lambda} &= \phi(r_{\lambda},r_{\lambda+1}) \left( 2r_{\lambda} dr_{\lambda}(J-) + 2r_{\lambda+1} dr_{\lambda+1}(J-) \right) \\
    &= \phi(r_{\lambda},r_{\lambda+1}) \left(2 r_{\lambda}^2 d \theta_{\lambda} + 2 r_{\lambda+1}^2 d \theta_{\lambda+1} \right)
\end{align*}
where
\[\phi(x,y) = \frac{1}{1 + x^2 + y^2} \]

Now, consider the change of coordinates from $\{z_{\lambda-1} = 1\}$ to $\{z_{\lambda+1} = 1\}$:
\begin{align*}
    s_{\lambda} &= \frac{r_{\lambda}}{r_{\lambda+1}} & \psi_{\lambda} &= \theta_{\lambda} - \theta_{\lambda+1} \\ s_{\lambda-1} &= \frac{1}{r_{\lambda+1}} & \psi_{\lambda-1} &= - \theta_{\lambda+1}
\end{align*} 
Then
\begin{align*}
    \alpha_{\lambda-1} &= \phi(s_{\lambda-1},s_{\lambda}) ( 2 s^2_{\lambda-1}d \psi_{\lambda-1} + 2 s^2_{\lambda}d \psi_{\lambda}) \\
    &= r^2_{\lambda+1}\phi(r_{\lambda},r_{\lambda+1})\left(-2 \frac{1}{r^2_{\lambda+1}} d \theta_{\lambda+1} + 2 \frac{r^2_{\lambda}}{r^2_{\lambda+1}} (d \theta_{\lambda} - d \theta_{\lambda+1}) \right) \\
    &= \phi(r_{\lambda},r_{\lambda+1}) \left( -2 d \theta_{\lambda+1} + 2 r^2_{\lambda}( d \theta_{\lambda} - d \theta_{\lambda+1} )\right)
\end{align*}
So that
\begin{align*}
    \alpha_{\lambda} - \alpha_{\lambda-1} &= \phi \left(2 r^2_{\lambda+1} d \theta_{\lambda+1} + 2 d \theta_{\lambda+1} + 2 r^2_{\lambda} d \theta_{\lambda+1} \right) = 2 r^2_{\lambda} d \theta_{\lambda}
\end{align*}
Therefore, along $H_{\lambda} = Z_{\lambda} \cap Z_{\lambda-1} = \{r_{\lambda} = 1\}$, we obtain the closed 1-form $\beta_{\lambda} = 2 d \theta_{\lambda}$.  In summary, by tracing through the $\check{\text{C}}$ech-DeRham isomorphism, we can replace a closed 2-form with a 1-cocycle of closed 1-forms:
\[\omega_{FS} \mapsto (2 d\theta_1, 2 d \theta_2, 2( -d \theta_1 - d \theta_2))\]
Now consider the restriction to the central torus.  The pair $\{d \theta_1, d \theta_2\}$ define a framing of $T^*T^2$.  Moreover, they are Hodge dual for some flat metric $C$ on $T^2$ that makes this an orthonomal framing.  The three 1-forms induce three kernel foliations $\cF_1,\cF_2,\cF_3$ that are linear with respect to the coordinates $\theta_1,\theta_2$.  The condition $d \theta_{\lambda} \wedge d \theta_{\lambda+1} > 0$ implies that these foliations are pairwise transverse.

\subsection{Branched covering}

Now consider a projective surface $X$ with a holomorphic branched covering
\[p: X \rightarrow \CP^2\]
that is generic with respect to the standard trisection $\CP^2$.  Such a branched covering always exists by choosing an embedding $X \hookrightarrow \CP^N$ and projecting onto a generic projective plane.  We can then pull back the standard trisection via $p$ to get a decomposition 
\[X = \widetilde{Z}_1 \cup \widetilde{Z}_2 \cup \widetilde{Z}_3\]
into three analytic polyhedra.  In general, the smooth decomposition is not a trisection, as the sectors may have 2-handles.  However, it accurately illustrates the geometric picture in the holomorphic setting.

Let $\cK \subset X$ be the ramification locus of $\pi$.  Assume that $p(\cK)$ is immersed and self-transverse and in general position with respect to the stratification induced by the trisection.  In particular, it intersects the central $T^2$ in $2b$ points and each $H_{\lambda} = S^1 \times \DD^2$ along arcs.  Positivity of intersection implies that these arcs are transverse to the foliation by holomorphic disks.

Now, consider the sector $Z_{\lambda}$ in $\CP^2$.  It lies in the open chart $U_{\lambda} = \CP^2 \setminus L_{\lambda}$, where $L_{\lambda} = \{z_{\lambda-1} = 0\}$ denotes the projective line at infinity with respect to this chart.  In addition, it is the analytic polyhedron defined by the holomorphic coordinate functions $z_{\lambda},z_{\lambda+1}$ in the chart $U_{\lambda}$.

The preimage $\widetilde{U}_{\lambda} = p^{-1}(U_{\lambda})$ is the complement of a hyperplane divisor in $X$, therefore it is an open Stein surface in $X$.    The preimage $\widetilde{Z}_{\lambda} = p^{-1}(Z_{\lambda})$ is an analytic polyhedron
\[\widetilde{Z}_{\lambda} = \{w \in \widetilde{U}_{\lambda} : \text{max}\left(|z_{\lambda} \circ p(w)|,|z_{\lambda+1} \circ p(w)| \right)  \leq 1 \}\]
in $\widetilde{U}_{\lambda}$, where $f_1 = z_{\lambda} \circ p$ and $f_2 = z_{\lambda+1} \circ p$ are holomorphic functions on the Stein surface $\widetilde{U}_{\lambda}$.  

The boundary of $\widetilde{Z}_{\lambda}$ is piecewise-smooth, with a corner at $p^{-1}(T^2)$.  Over the codimension-0 faces $\widetilde{H}_{\lambda}$ and $\widetilde{H}_{\lambda+1}$, the boundary is Levi-flat and foliated by compact holomorphic curves.  In particular, 
\begin{align*}
    \widetilde{H}_{\lambda} &= \bigcup_{\theta} \left\{ w \in \widetilde{U}_{\lambda} : z_{\lambda} \circ p (w) = e^{2 \pi i \theta} \right\} \cap \widetilde{Z}_{\lambda}\\
    \widetilde{H}_{\lambda+1} &= \bigcup_{\theta} \left\{ w \in \widetilde{U}_{\lambda} : z_{\lambda+1} \circ p (w) = e^{2 \pi i \theta} \right\} \cap \widetilde{Z}_{\lambda}
\end{align*}
If the branch locus $p(\cK)$ intersects $H_{\lambda}$ transversely, then these curves are nonsingular.

Next, the branch locus $p(\cK)$ intersects $T^2$ transversely.  Since $T^2$ is nullhomologous, the signed intersection of $p(\cK)$ with $T^2$ vanishes and there exist $b$ positive intersections and $b$ negative intersections for some $b$.  If $p$ has degree $n$ as a finite holomorphic map, then $p^{-1}(T^2)$ is a simple $n$-fold branched covering of $T^2$ over the $2b$ points.  In general, it may be disconnected.  In $\CP^2$, the torus $T^2$ is Lagrangian with respect to the Fubini-Study K\"ahler form.  The pullback $p^*(\omega_{FS})$ also vanishes along $\Sigma = p^{-1}(T^2)$, although this form is not symplectic.  Its square vanishes precisely along the ramification locus $\cK$.  To modify it to a symplectic form, it is necessary to add an exact 2-form $d \phi$ as in \cite[Proposition 10]{Auroux}.  The restriction of this 2-form to $\Sigma$ will have nonzero support in a neighborhood of the ramification locus $B = \Sigma \cap \cK$.  Consequently, this surface is no longer Lagrangian.  

Finally, a flat metric $g$ on $T^2$ pulls back to a singular flat metric on $\Sigma$.  The closed 1-forms $(d \theta_2, -d \theta_1, d \theta_1 - d \theta_2)$ pull back to harmonic 1-forms with respect to the conformal structure determined by the singular flat metric.  Moreover, the form $p^*(\beta_{\lambda})$ extends across $H_{\lambda}$ as the 1-form defining the Levi-flat foliation by holomorphic curves.  The fact that 
$$p^*(\beta_{\lambda}) \wedge p^*(\beta_{\lambda+1}) = p^*(\beta_{\lambda} \wedge \beta_{\lambda+1}) \geq 0$$
implies that these foliations meet transversely as well.  Thinking 4-dimensionally, this transversality condition is equivalent to the statement that $\Sigma$ is a convex corner of the analytic polyhedron $\widetilde{Z}_{\lambda}$.

The singular points $\Sigma \cap \cK$ are the preimages of the branch points on $T^2$.  We can describe this branching in local coordinates by identifying $T^2$ with $i \RR^2 \subset \CC^2$ and consider the complex surface
\[\{z^2 = w_1 \mp i w_2\} \subset \CC^3\]
for a positive (resp. negative) intersection point.  The surface $\Sigma$ is the intersection of $X$ with the subset $\{\text{Re}(w_1) = \text{Re}(w_2) = 0\}$.  The tangent plane at $0 \subset \Sigma$ is an isolated complex point of $\Sigma$, meaning that $T_0 \Sigma$ is a complex line with respect to the almost-complex structure on $X$ (and not simply a copy of $\RR^2$).  Complex lines come with an orientation, therefore the singular point is oriented and this agrees with the orientation of $p(\cK) \cap T^2$ downstairs.  But since this orientation is induced by the complex structure, it must extend in a compatible way with the foliation on $\widetilde{H}_{\lambda}$ by holomorphic curves.

Finally, the boundary of $\widetilde{Z}_{\lambda}$ can be smoothed to obtain a Stein-fillable contact structure.  The defining function $\phi_{\lambda} = \text{max}(|f_1|,|f_2|)$ is a continuous, plurisubharmonic function on $\widetilde{U}_{\lambda}$.   It can be $C^0$-approximated by a smooth, strictly plurisubharmonic function $\widetilde{\phi}_{\lambda}$.  The level set $\widetilde{Y}_{\lambda} = \widetilde{\phi}^{-1}_{\lambda}(1)$ naturally admits a Stein-fillable contact structure.

 
\section{Handlebody setup}

In order to construct the required symplectic structure over each handlebody $H_{\lambda}$, we need to construct an exact 2-form $d \mu_{\lambda}$ that orients the leaves of the foliation $\cF_{\lambda} = \text{ker}(\betat_{\lambda})$.

\subsection{Morse form foliations}

We restrict attention to foliations that are defined as the kernel of a smooth, closed 1-form $\beta$.  In particular, a closed 1-form is {\it Morse} if near each singular point, there exists local coordinates and a function $f$ that is quadratic in these coordinates such that $\beta = df$.  Moreover, the closed 1-form $\beta$ is {\it intrinsically harmonic} if it is harmonic for some Riemannian metric.

We can apply various results on the topology of Morse form foliations using the following doubling argument.  Recall that $\betat_{\lambda} = \beta_{\lambda} + d(sf)$ near the boundary.  We can take the double $\#_g S^1 \times S^2$ of $H_{\lambda}$ and extend $\betat_{\lambda}$ as a closed 1-form $D(\betat_{\lambda})$ on the total space by assuming
\[D(\betat_{\lambda}) = \beta_{\lambda} + d(s^2f)\]
on a tubular neighborhood of $\Sigma$.  This 1-form has Morse singularities at the singular points $B \subset \Sigma$ that are of index 1 or 2, according to the sign of $f$.  In addition, the induced foliation is topologically conjugate to $\cF_{\lambda}$.

There are significant restrictions on the topology of a foliation defined by a closed, Morse 1-form.  For example, the foliation has trivial holonomy.  The leaves of a closed Morse 1-form are either {\it compact} or {\it noncompact}.  A finite number of {\it singular} leaves intersect the vanishing locus of the 1-form, while all remaining leaves are {\it nonsingular}.  Given a closed Morse 1-form on a compact $n$-manifold $M$, there is a decomposition $M = M_c \cup M_{\int}$ into two compact $n$-manifolds such that the interior $M_c$ is the union of all compact leaves, the interior of $M_{\infty}$ is the union of all noncompact leaves, and $M_c \cap M_{\infty} = \partial M_c \cap \partial M_{\infty}$ is a subvariety that is the union of some compactified components of noncompact, singular leaves \cite[Proposition 1]{FKL}.

A {\it closed transversal} to a foliation is closed, smooth curve that is everywhere positively-transverse to the foliation. 

\subsection{Tautness of the foliation}

Recall that a foliation $\cF$ on $M$ is {\it taut} if for every point $p \in M$, there is a closed transversal to $\cF$ through $p$.  In this subsection, we will prove that the foliation $\cF_{\lambda}$ on $H_{\lambda}$ satisfy the following stronger version of tautness.

\begin{proposition}
\label{prop:taut-annulus}
For each point $p \in H_{\lambda}$, there exists 
\begin{enumerate}
    \item an embedded, closed transversal $\gamma_p$ to $\cF_{\lambda}$ through $p$,
    \item an embedded annulus $A_p \cong S^1 \times [0,1]$ consisting of a family of embedded, closed transversals to $\cF_{\lambda}$ such that the unoriented boundary is $\partial A_p = \gamma_p \cup \gamma$, where $\gamma$ is an embedded curve in $\partial H_{\lambda} = \Sigma$.
\end{enumerate}
\end{proposition}

The proposition follows easily from the following two results.

\begin{lemma}
\label{lemma:Calabi}
Through every point $q \in \Sigma \setminus B$, there exists an embedded curve $\gamma \subset \Sigma$ that is a closed transversal to the kernel foliation of $\beta_{\lambda}$.
\end{lemma}

\begin{proof}
By assumption, $\beta_{\lambda}$ is a harmonic 1-form on $\Sigma$ and therefore Calabi's condition ensures there is an immersed closed transversal through each nonsingular point \cite{Calabi}.  By a perturbation, we can assume that the curve is self-transverse and resolve the self-intersections preserving the fact that is a (possibly disconnected) embedded closed transversal.  If this closed transversal is disconnected, choose the component through $p$.
\end{proof}

\begin{proposition}
\label{prop:all-leaves-touch}
Every leaf of $\cF_{\lambda}$ intersects the boundary $\partial H_{\lambda} = \Sigma$.
\end{proposition}

This follows immediately by combining Lemmas \ref{lemma:compact-intersects-boundary} and \ref{lemma:noncompact-intersects-boundary}.  We defer these proofs until after the following corollary and the proof of the main proposition.

\begin{corollary}
The doubled 1-form $D(\betat_{\lambda})$ on $D(H_{\lambda})$ is intrinsically harmonic.
\end{corollary}

\begin{proof}
Since each leaf intersects $\Sigma$ and $\beta_{\lambda}$ is intrinsically harmonic on $\Sigma$, then $D(\betat_{\lambda})$ satisfies Calabi's condition.
\end{proof}

\begin{proof}[Proof of Proposition \ref{prop:taut-annulus}]
The boundary foliation admits an embedded closed transversal that intersects every leaf (Lemma \ref{lemma:Calabi}).  Since each leaf of $\cF_{\lambda}$ is connected to the boundary (Proposition \ref{prop:all-leaves-touch}), we can find $\gamma_p$ by a leaf-preserving isotopy.  The trace of this isotopy is the required annulus $A_p$.
\end{proof}

We now address the case of compact leaves.

\begin{lemma}
\label{lemma:no-tori}
The foliation $\cF_{\lambda}$ on $H_{\lambda}$ has no $T^2$ leaves.
\end{lemma}

\begin{proof}
Every torus in a 1-handlebody is compressible and moreover bounds a solid torus $Q$.  Then $[\betat_{\lambda}] = [0] \in H^1_{DR}(Q,\partial Q)$ and so there is a function $\phi: Q \rightarrow \RR$ such that $\phi(\partial Q) = 0$ and $\betat_{\lambda} = d \phi$ on $Q$.  This function attains a maximum at some point $x \in Q$ and therefore $(\betat_{\lambda})_x = d \phi_x = 0$, which is a contradiction.
\end{proof}

Next, we can use the Euler class of the foliation $\cF_{\lambda}$ to show that every closed, compact leaf must be an (excluded) torus.  Using the doubling argument, we get isomorphisms
\begin{align*}
    H_1(H,B) &\cong H_1(Y,B) \\
    H_2(H,\Sigma \setminus B) &\cong H_2(Y \setminus B)
\end{align*}
We can therefore use the intersection pairing
\[H_1(Y,B) \times H_2(Y \setminus B) \rightarrow H_0(Y \setminus B) \cong H_0(Y) \cong \ZZ\]
to define an intersection pairing
\[ \langle , \rangle: H_1(H_{\lambda},B) \times H_2(H_{\lambda}, \Sigma \setminus B) \rightarrow \ZZ \]

Choose a vector field $v_{\lambda}$ directing the kernel of $\beta_{\lambda}$ on $\Sigma$ and extend it to a section of $\text{ker}(\betat_{\lambda})$ over $H_{\lambda}$.  Generically, this section will vanish along a smooth 1-complex representing a well-defined class Euler class $e(\cF_{\lambda}) \in H_1(H_{\lambda},B)$ of the foliation.

\begin{lemma}
\label{lemma:Euler}
Let $L \in H_{\lambda}$ be a compact leaf $L$ of $\cF_{\lambda}$.  Then the Euler characteristic of $L$ is given by
\[\langle e(\cF_{\lambda}), [L] \rangle \]
\end{lemma}

\begin{proof}
A section $s$ of $\text{ker}(\beta_{\lambda})$, restricted to $L$, is a vector field on $L$ and the Euler characteristic of $L$, which is therefore given by the signed count of the zeros of $s|_L$, is given by the intersection of $e(\cF_{\lambda})$ with $L$.
\end{proof}

\begin{lemma}
\label{lemma:compact-intersects-boundary}
Every compact leaf of the foliation $\cF_{\lambda}$ intersects $\Sigma$.
\end{lemma}

\begin{proof}
If $L$ is a compact, closed, co-oriented hence oriented leaf, then by Lemma \ref{lemma:Euler} its Euler characteristic vanishes since $[L] = 0 \in H_2(H_{\lambda},\Sigma \setminus B)$.  But this implies it is a torus, which is excluded by Lemma \ref{lemma:no-tori}.
\end{proof}

We now finish the proof of Proposition \ref{prop:all-leaves-touch} by excluding noncomapct leaves.

\begin{lemma}
\label{lemma:noncompact-intersects-boundary}
Every noncompact leaf of the foliation $\cF_{\lambda}$ intersects $\Sigma$.
\end{lemma}

\begin{proof}
The boundary of the closure of a noncompact leaf $L$ contains a compactified component $L'$ of a singular leaf (by \cite[Proposition 1]{FKL}).  The set of accumulation points of $L$ in $L'$ is open in $L'$; in particular, the closure of the set of accumulation points contains the singular point.  The singularities all lie in $\Sigma$, hence $L'$ intersects $\Sigma$ and in particular it intersects $\Sigma$ along a compactified component $\widetilde{L}'$ of a singular, noncompact leaf of the kernel foliation of $\beta_{\lambda}$ on $\Sigma$.  

Choose some point $x$ in $\widetilde{L}'$.  Every neighborhood of $x$ contains $L$ and moreover, since $x \in \Sigma \setminus B$, the foliation is locally a product with the foliation on $\Sigma$.  In particular, $L$ must intersect this neighborhood as a leaf intersecting the boundary.
\end{proof}

\subsection{Thurston norm}

Just as compact leaves of a taut foliation are known to minimize the Thurston norm \cite{Thurston}, the nonsingular compact leaves of the foliation $\cF_{\lambda}$ minimize the (generalized) Thurston norm defined by Scharlemann \cite{Scharlemann}.  

To prove this, we first establish an embedding result that reduces the problem to the original case.
\begin{proposition}
\label{prop:foliated-handlebody-embedding}
Let $M$ be a compact 3-manifold with boundary and let $f: M \rightarrow S^1$ be a smooth map such that
\begin{enumerate}
    \item $f$ has no critical points in $M$, and
    \item $f|_{\partial M}$ has only nondegenerate critical points of index 1.
\end{enumerate}.  
Then there exists a fibered 3-manifold $\pi: Y \rightarrow S^1$ and an embedding $i: M \hookrightarrow Y$ such that $\pi \circ i = f$.
\end{proposition}

\begin{proof}
Without loss of generality, assume that $0$ is a regular value of $f|_{\partial M}$.  Let $F_0 = f^{-1}(0)$ be the compact, regular level set.  As $\theta$ varies in $[0,2\pi]$, the topology of the regular level set $f^{-1}(\theta)$ changes at the critical points of $f|_{\partial M}$ by either adding or removing a 1-handle.  Fix a Morse-Smale metric $g$ on $\Sigma$.  If $p \in \text{Crit}(f|_{\partial M})$ corresponds to adding a 1-handle to the regular level set, then we can use the unstable manifold of $p$ to isotope the attaching $S^0$ for this down to $\partial F_0$.  On the other hand, if $p$ is a critical point corresponding to removing a 1-handle, then we can use the stable manifold of $p$ to isotope the attaching $S^0$ for the cancelling 1-handle up to $F_{ 2\pi}$.

We therefore have two surfaces
\[\widehat{F}_0 = F_0 \cup \{\text{1-handles}\}\qquad \widehat{F}_{2\pi} = F_{2 \pi} \cup \{\text{1-handles}\}\]
Let $M^0$ be the compact 3-manifold obtained by cutting $M$ along $f^{-1}(0)$.  Then $M^0$ embeds into $\widehat{F}_0 \times [0,2\pi]$ in a fiber-preserving manner.  Similarly, it also embeds into $\widehat{F}_{2\pi} \times [0,2\pi]$.

Now choose a closed surface $F$ and embeddings $\widehat{F}_0, \widehat{F}_{2\pi} \hookrightarrow F$.  To obtain $Y$, we choose any homeomorphism $\phi: F \rightarrow F$ such that $\phi(F_{2 \pi}) = F_{0}$ and take the mapping torus
\[ Y = F \times [0,2\pi]/ (x,2\pi) \sim (\phi(x),0)\]
The embedding $M_0 \subset \widehat{F}_0 \times [0,2 \pi] \subset F \times [0,2\pi]$ then descends to an embedding of $M$ into $Y$.
\end{proof}

\begin{theorem}
Suppose that $\betat_{\lambda}$ is integral.  Then each nonsingular leaf $L$ minimizes the Thurston norm for the class $[L] \in H_2(H_{\lambda},\partial L)$.
\end{theorem}

\begin{proof}
Since $\betat_{\lambda}$ is integral, there exists a map $f: H_{\lambda} \rightarrow S^1$ such that $f^*(d \theta) = \betat_{\lambda}$.  By Proposition \ref{prop:foliated-handlebody-embedding}, we obtain an embedding $H_{\lambda} \hookrightarrow Y$ into a fibered 3-manifold that extends $f$.

Let $L = f^{-1}(\theta_0)$ be a nonsingular, compact leaf.  The $L$ embeds in the fiber $ \pi^{-1}(\theta_0)$.  Let $F_{\theta_0}$ denote the complement of $L$ in the fiber.  Let $L'$ be any embedded surface representing $[L] \in H_2(H_{\lambda},\partial L)$.  Then $L' \cup F_{\theta_0}$ is a closed surface representing the class of the fiber in $H_2(Y)$.  Since $\pi^{-1}(\theta_0)$ minimizes the Thurston norm in $Y$, we consequently have the inequality
\[-\chi(L \cup F_{\theta_0}) \leq -\chi(L' \cup F_{\theta_0})\]
which implies that $-\chi(L) \leq - \chi(L')$.  
\end{proof}

\subsection{Extending $B$ to a 1-complex $\cS$}

The construction of $\omega$ in the next subsection requires a small perturbation near the vanishing locus $B \subset \Sigma$. To extend this perturbation across the spine, it is necessary to extend $B$ to a 1-complex $\cS$ that is transverse to the foliations $\{\cF_{\lambda}\}$.  We can accomplish this via the following lemma.

\begin{lemma}
\label{lemma:transverse-paths}
For each $\lambda$, there exists a collection of disjoint, embedded paths $\gamma_{\lambda,1},\dots,\gamma_{\lambda,g-1}$ in $H_{\lambda}$ with boundary on $B$ that are transverse to the foliation $\cF_{\lambda}$.
\end{lemma}

For brevity, let $\gamma_{\lambda}$ denote the union $\cup \gamma_{\lambda,i}$ of the $g-1$ arcs constructed in $H_{\lambda}$. Then the 1-complex is the union
\[ \cS := \bigcup_{\lambda = 1,2,3}\gamma_{\lambda}\]

\begin{proof}[Proof of Lemma \ref{lemma:transverse-paths}]
Choose a pairing on the points of $B$, where each pair has one positive and one negative point (with respect to the orientation on $B$ induced by the function $f$).  In particular, we can index the positive points as $p_1,\dots,p_{g-1}$ and the negative points as $q_1,\dots,q_{g-1}$.

By assumption, the 1-form $\beta_{\lambda}$ is harmonic on $\Sigma$.  Calabi has shown that, consequently, for any two points $p_i,q_i \in \Sigma$, there is an embedded path transverse to $\text{ker}(\beta_{\lambda})$ in $\Sigma$ from $q_i$ to $p_i$ \cite{Calabi}.  Let $\gamma_{\lambda,1},\dots,\gamma_{\lambda,g-1}$ be a collection of such transverse paths.  These paths can be pushed into $H_{\lambda}$ as a collection of embedded paths transverse to $\cF_{\lambda}$.
\end{proof}

By the Morse lemma, we can choose local coordinates $(x,y)$ near each $b \in B$ such that 
\[\beta = xdx - ydy \qquad \ast \beta = x dy + y dx\]
This further implies that
\[\beta \wedge \ast \beta = (x^2 + y^2) dx \wedge dy\]
We set $R = x^2 + y^2$ to be the square of the radius. Then near $B$ we have
\[\beta_1 \wedge \beta_2 = \beta_2 \wedge \beta_3 = \beta_3 \wedge \beta_1 = R dx \wedge dy\]

We now extend these local coordinates to coordinates on each cylinder $\nu(\gamma_{\lambda,i})$. The foliation $\cF_{\lambda}$ induces a foliation by disks on each $\nu(\gamma_{\lambda,i})$ and therefore determines a product splitting $\nu(\gamma_{\lambda,i}) = \gamma_{\lambda,i} \times D^2$.  We can view $(x,y)$ as local coordinates on the leaves of $\cF_{\lambda}$ that intersect $\gamma_{\lambda,i}$.   By construction, each arc $\gamma_{\lambda,i}$ of $\cS$ connects two points where $f$ takes opposite signs.  At the endpoints of $\gamma_{\lambda,i}$, the surface $\Sigma$ inherits two orientations: one as the boundary of $H_{\lambda}$ and the other as the boundary of $\nu(\gamma_{\lambda,i})$.  We can assume these agree at the endpoint where $f = 1$ and disagree at the endpoint where $f = -1$.  In particular, at the negative endpoint of $\gamma_{\lambda,i}$, the 2-form $dx \wedge dy$ induces the negative orientation on $\Sigma$.

In the discussion that follows, we will let $R = x^2 + y^2$ in this local coordinate system and let $D_R$ denote the disk of radius $\sqrt{R}$.  We also set constants
\[0 < R_0 < R_1 < R_2\]
where $\sqrt{R_2}$ is less than the radius of definition of the coordinates $(x,y)$ on $\nu(\cS)$.  We will let $\nu_R(B)$ and $\nu_R(\cS)$ denote the tubular neighborhoods of radius $\sqrt{R}$ around $B$ in $\Sigma$ and $\cS$ in $H_{\lambda}$, respectively.

\subsection{Handlebody primitive}  

The goal of this section is to construct an exact 2-form on $H_{\lambda}$ that orients the leaves of $\cF_{\lambda}$.  Since $\cF_{\lambda}$ is much like a taut foliation, we adapt a standard argument (eg \cite[Section 4.4]{Calegari}) for constructing 2-forms that orient taut foliations.  Throughout, we will consider a collar neighborhood $\Sigma \times [0,8\epsilon]$ of $\partial H_{\lambda}$, where $\epsilon$ is an undefined, small constant to be chosen later and $s$ is a normal coordinate to $\Sigma = \partial H_{\lambda}$.

\begin{lemma}
\label{lemma:q-paths}
There exists a collection of paths $\{\gamma_{q_1},\dots,\gamma_{q_k}\}$ in $\Sigma \setminus B$ and closed 1-forms $PD(\gamma_{q_1}),\dots,PD(\gamma_{q_k})$ such that
\begin{enumerate}
    \item $\beta_{\lambda} \wedge PD(\gamma_{q_i}) \geq 0$ for all $i = 1,\dots,k$,
    \item the union $\bigcup \text{Supp}(\beta_{\lambda} \wedge PD(\gamma_{q_i}))$ of the supports covers the complement of $\nu_{R_0}(B)$.
\end{enumerate}
\end{lemma}

\begin{proof}
Let $q$ be a point in the complement of a $\sqrt{R}_0$-neighborhood of $B$.  Since $\beta_{\lambda}$ is harmonic, we can find a closed transversal through $q$.  A tubular neighborhood $\nu(\gamma_q)$ splits as a product $\gamma_q \times [-\delta,\delta]$.  Let $s$ be a normal coordinate and $\phi(s)$ a bump function supported on $[-\delta,\delta]$.  Define $PD(\gamma_q) = \phi(s)ds$ in these local coordinates.  Since $\gamma_q$ is a closed transversal, we have that $\beta_{\lambda} \wedge PD(\gamma_q) \geq 0$.

The set $\{\text{Supp}(\beta_{\lambda} \wedge PD(\gamma_q))\}$ as $q$ ranges through the points in $\Sigma \smallsetminus \nu_{\sqrt{R}_0} (B)$ is an open cover of a compact set, hence it has a finite subcover that is indexed by a finite set of points $q_1,\dots,q_k$. 
\end{proof}

Choose a cutoff function satisfying
\[\psi_q = \begin{cases} 1 & \text{if $s \geq 5 \epsilon$} \\ 0 & \text{if $s \leq 2\epsilon$}\end{cases} \]
such that $\psi'_q > 0$ for $s \in [3\epsilon,5\epsilon]$.  Define the 1-form
\[\mu_q = \psi_q \sum_{i=1}^k PD(\gamma_{q_i})\]
where $PD(\gamma_{q_i})$ also denotes (by abuse of notation) the pullback of $PD(\gamma_{q_i})$ via the projection map $\pi: \Sigma \times [0,8\epsilon] \rightarrow \Sigma$.

\begin{lemma}
For $\epsilon > 0$ sufficiently small, we have that $\betat_{\lambda} \wedge d\mu_q \geq 0$ and is strictly positive when $s \in [3\epsilon,5\epsilon]$ and $R \geq R_0$.
\end{lemma}

\begin{proof}
We have that
\begin{align*}
    \betat_{\lambda} \wedge d\mu_q &= \sum_{i = 1}^k (\beta_{\lambda} + f ds + s df) \wedge \psi'_q ds \wedge PD(\gamma_{q_i}) \\
    &= -\psi'_q \sum \left( ds \wedge \beta_{\lambda} \wedge PD(\gamma_{q_i}) + s ds \wedge df \wedge PD(\gamma_{q_i}) \right)
\end{align*}
The derivative $-\psi'_q$ is positive and the first term in the sum is nonnegative everywhere and strictly positive on the support of $PD(\gamma_{q_i})$.  Moreover, it dominates the second term in the sum for $|s| \leq 5 \epsilon$ sufficiently small.
\end{proof}

We can choose an annulus $A_{q_i}$ with boundary $\gamma_{q_i} \times \{ 4\epsilon\} \cup \gamma'_{q_i} \times \{0\}$ such that
\begin{enumerate}
    \item $\gamma'_{q_i}$ is transversely isotopic to $\gamma_{q_i}$ in $\Sigma \setminus B$ and disjoint from an $\sqrt{R}_2$-neighborhood of $B$,
    \item the intersection $A_{q_i} \cap \Sigma \times [0,8 \epsilon]$ is the union of $\gamma_{q_i} \times [4\epsilon,8\epsilon]$ and $\gamma'_{q_i} \times [0,8\epsilon]$.
\end{enumerate}
We can then use the following lemma to extend $\mu_{q_i}$ to a 1-form supported in a neighborhood of $A_{q_i}$.

\begin{lemma}
\label{lemma:mu-p-construction}
Let $p \in H_{\lambda}$ be a point and let $\gamma_p,A_p$ be a closed transversal and annulus satisfying the conclusions of Proposition \ref{prop:taut-annulus}.  There exist tubular neighborhoods $U_p \subset U'_p$ of $\gamma_p$ and $V_p$ of $A_p$ and an exact 2-form $d \mu_p$ such that
\begin{enumerate}
    \item the support of $\mu_p$ is contained in $\nu(A_p)$,
    \item the form $\betat_{\lambda} \wedge d \mu_p$ is positive on $U_p$ and its positive support is contained in $U'_p$,
    \item the negative support of $\betat_{\lambda} \wedge d \mu_p$ is contained in $\nu(A_p) \cap \Sigma \times [\epsilon,2\epsilon]$
\end{enumerate}
\end{lemma}

\begin{proof}
We can choose coordinates $(s,t,\theta)$ on $\nu(A_p) \cong [0,1+2\delta] \times [-\delta,\delta] \times S^1$ so that
\begin{enumerate}
    \item $s$ agrees with the collar coordinate on $\partial H_{\lambda}$ for $0 \leq s \leq 6 \epsilon$,
    \item $\betat_{\lambda} = d \theta$
\end{enumerate}
Let $\phi_1(t)$ be a bump function on $[-\delta,\delta]$ and choose a second bump function satisfying
\[\phi_2(s) = \begin{cases} 0 & \text{for $s \geq 1 + \delta$} \\ 1 &\text{for $ 1 - \delta \geq s \geq 2\epsilon$} \\ 0 & \text{for $s \leq \epsilon$}\end{cases}\]
and $\phi'_2 > 0$ on the interval $1 - \delta/2 \leq s \leq 1 + \delta/2$.
Define
\[\mu_p = \phi_2(s)\phi_1(t) dt\]
Then
\[d \mu_p = \phi_1 \left( \phi'_2 ds \wedge dt \right)\]
and
\[ \betat_{\lambda} \wedge d\mu_p = \phi_1 \phi'_2 \betat_{\lambda} \wedge ds \wedge dt\]
This is nonzero precisely on the support of $\phi_1 \phi'_2$.  The function $\phi_1$ is nonnegative everywhere, while $\phi'_2$ is positive on the interval $1 - \delta/2 \leq s \leq 1 + \delta/2$, the positive support of $\phi'_2$ is contained in the interval $1-\delta \leq s \leq 1 + \delta$ and the negative support is contained in the interval $ \epsilon \leq s \leq 2\epsilon$.  We can therefore choose $U_p = S^1 \times [1-\delta/2,1+\delta/2] \times [-\delta/2,-\delta/2]$ and $U'_p = S^1 \times [1-\delta,1+\delta] \times [-\delta,\delta]$.
\end{proof}

Define
\[W_{\epsilon} = H_{\lambda} \setminus (\Sigma \times [0,5 \epsilon])\]

\begin{proposition}
\label{lemma:p-paths}
There exists a collection of points $\{p_1,\dots,p_k\}$ in $W_{\epsilon}$ such that
\begin{enumerate}
    \item $\betat_{\lambda} \wedge d \mu_{p_i} \geq 0$ for all $i=1,\dots,k$,
    \item the union $\bigcup_{i = 1}^k \text{Supp}(\betat_{\lambda} \wedge d \mu_{p_i})$ of the supports covers $W_{\epsilon}$.
    \item the 3-form
    \[\betat_{\lambda} \wedge \sum_{i=1}^k d\mu_{p_i}\]
    is a volume form on $W_{\epsilon}$
\end{enumerate}
\end{proposition}

\begin{proof}
For each $p \in W_{\epsilon}$, we can find some closed transversal $\gamma_p$ and annulus $A_p$ and use Lemma \ref{lemma:mu-p-construction} to define $\mu_p$.  As in the proof of Lemma \ref{lemma:q-paths}, the union of the supports of $d \mu_p$ as $p$ ranges over the set $W_{\epsilon}$ is an open cover, which by compactness has a finite subcover indexed by a finite collection of points $p_1,\dots,p_n$.  The required properties follow immediately by the construction of Lemma \ref{lemma:mu-p-construction}.
\end{proof}

We now define the 1-form
\begin{equation}
    \label{eq:mu-H-definition}
    \mu_{H_,\lambda} = \mu_q + \sum_{i = 1}^k \mu_{p_i}
\end{equation}

\begin{lemma}
For $C > 0$ sufficently large,
\begin{enumerate}
    \item $\alpha_+ = C \betat_{\lambda} + \mu_{H,\lambda}$ is a positive contact form on $W_{\epsilon}$,
    \item $\alpha_- = -C \betat_{\lambda} + \mu_{H,\lambda}$ is a negative contact form on $W_{\epsilon}$.
\end{enumerate}
\end{lemma}

\begin{proof}
By construction,
\[\betat_{\lambda} \wedge d \mu_{H,\lambda} = \betat_{\lambda} \wedge d \mu_q + \betat_{\lambda} \wedge \sum d \mu_{p_i} > 0\]
Furthermore, we have that
\[\alpha_+ \wedge d \alpha_+ = C \betat_{\lambda} \wedge d\mu_{H,\lambda} + \mu_{H,\lambda} \wedge d \mu_{H,\lambda}\]
and
\[\alpha_- \wedge d \alpha_- = C( - \beta_{\lambda} \wedge d \mu_{H,\lambda}) + \mu_{H,\lambda} \wedge d \mu_{H,\lambda}\]
In both cases, the first term dominates for $C > 0$ sufficiently large.
\end{proof}

\section{Constructing the form}

In this section, we describe the symplectic form $\omega$ in a neighborhood of each stratum of the trisection decomposition and how to patch these forms together to get a global symplectic form.

\subsection{Coordinates on the spine}
\label{sub:coordinates}

Let $(p_1,p_2)$ denote coordinates on $D^2$ and let $(x,y)$ denote local coordinates on $\Sigma$.  We assume that locally $dx \wedge dy$ orients the surface $\Sigma$ and the orientation on $\nu(\Sigma) = D^2 \times \Sigma$ is given by the 4-form
\[dp_1 \wedge dx \wedge dp_2 \wedge dy = dp_2 \wedge dp_1 \wedge dx \wedge dy\]

In a neighborhood of the central surface, we can assume that locally in $\Sigma \times D^2$ the pieces of the trisection are given by:
\begin{align*}
    H_1 &= \{p_1 \geq 0, p_2 = 0\} &Z_1 &= \{ p_1 \geq 0, p_2 \leq 0\} \\
    H_2 &= \{p_1 + p_2 = 0, p_2 \geq 0\} &     Z_2 &= \{ p_2 \geq 0, p_1 + p_2 \geq 0\} \\
    H_3 &= \{p_1 = 0, p_2 \leq 0\} &    Z_3 &= \{ p_1 \leq 0, p_1 + p_1 \leq p_1\}
\end{align*}
In particular, if we set
\[\phi_1 = p_2 \qquad \phi_2 = -p_1-p_2 \qquad \phi_3 = p_1\]
then 
\[H_{\lambda} = \{ \phi_{\lambda-1} \geq 0, \phi_{\lambda} = 0\} \qquad Z_{\lambda} = \{\text{max}(-\phi_{\lambda-1},\phi_{\lambda}) \leq 0 \}\] 
and in particular, $Z_{\lambda}$ is locally the sublevel set of $g_{\lambda} = \text{max}(\phi_{\lambda},-\phi_{\lambda-1})$.

We attach $H_{\lambda} \times I$ to $\Sigma \times D^2$ as follows.  Fix some $\epsilon > 0$.  Let $s \in [0,6\epsilon]$ be a coordinate on a collar neighborhood of $\partial H_{\lambda}$ and $t$ a coordinate on the $I$-factor. Glue $H_{\lambda} \times I$ to $\Sigma \times D^2$ using the following identifications for $\lambda = 1,2,3$:
\[ s = \phi_{\lambda-1} \qquad t = \phi_{\lambda}\]

Finally, we obtain a smooth manifold $\nu(H_1 \cup H_2 \cup H_3)$ with three smooth boundary components as follows.  Recall that the sector $Z_{\lambda}$ is locally defined near $\Sigma$ by $g_{\lambda} = 0$.  We can extend $g_{\lambda}$ across $H_{\lambda} \times I \cup H_{\lambda-1} \times I$.  Now let $\widehat{g}_{\lambda}$ denote some $C^0$-close smoothing of $g_{\lambda}$ (as a function of $p_1,p_2$), that agrees with $g_{\lambda}$ outside of a tubular neighborhood of $\Sigma$.  Let $\widehat{Y}_{\lambda}$ be the level set of $\widehat{g}^{-1}_{\lambda}(-\epsilon)$ for some small $\epsilon > 0$.  We can assume that the smoothing is chosen so that the form
\[ \zeta_{\lambda} = d \phi_{\lambda} - d \phi_{\lambda-1} = dt - ds \]
coorients $\widehat{Y}_{\lambda}$ in $\Sigma \times D^2$.

\subsection{Constructing $\omega$ in a neighborhood of $\Sigma$}

The first step is to construct a 2-form in a neighborhood of the central surface.

Let $(\beta_1,\beta_2,\beta_3)$ be the 1-cocycle of 1-forms.  By assumption, they have extensions $\betat_1,\betat_2,\betat_3$ across the three handlebodies that take the form
\begin{align*}
    \betat_{\lambda} &= \beta_{\lambda} + d(sf)
\end{align*}
near $\Sigma$ for a common function $f$.  We can extend these to 1-forms on the 4-dimensional $\Sigma \times D^2$ as 
\[\betat_{\lambda} = \beta_{\lambda} + d( \phi_{\lambda-1}f)\]
Define the starting form
\[\omega_0 = dp_2 \wedge \betat_1 - dp_1 \wedge \betat_2\]
and the 1-forms
\[\mu_{\Sigma,\lambda} = (t + 2) \betat_{\lambda} - (s - 1) \betat_{\lambda+1}\]
for $\lambda = 1,2,3$.

\begin{lemma}
For each $\lambda = 1,2,3$, we have that
\begin{align*}
    \omega_0 &= dt \wedge \betat_{\lambda} - ds \wedge \betat_{\lambda+1} \\
    \mu_{\lambda+1} &= \mu_{\lambda} - 3 \betat_{\lambda}
\end{align*}
Consequently, each 1-form $\mu_{\Sigma,\lambda}$ is a primitive for $\omega_0$.
\end{lemma}

\begin{proof}
This follows by a straightforward computation using the cocycle relations
\begin{align*}
    \phi_1 + \phi_2 + \phi_3 &= 0 \\
    \beta_1 + \beta_2 + \beta_3 &= 0
\end{align*}
\end{proof}

\begin{proposition}
\label{prop:first-symplectic-sigma}
In a sufficiently small neighborhood of $\Sigma$ we have
\begin{enumerate}
    \item $\omega_0 \wedge \omega_0 \geq 0$, and 
    \item for each $\lambda = 1,2,3$, we have that
    \[\zeta_{\lambda} \wedge \mu_{\Sigma,\lambda} \wedge \omega_0 \geq 0\]
    \item for each $\lambda = 1,2,3$, we have that
    \begin{align*}
        dt \wedge \mu_{\Sigma,\lambda} \wedge \omega_0 &\geq 0 \\
        dt \wedge \mu_{\Sigma,\lambda+1} \wedge \omega_0 & \leq 0
    \end{align*}
\end{enumerate}
In addition, all 4-forms vanish exactly along $B \times D^2$, where $B = \beta^{-1}(0) = \beta_{\lambda}^{-1}(0)$, and
\end{proposition}

\begin{proof}
First, by multiplying out, we have that
\begin{align*}
    \omega_0 \wedge \omega_0 &= -2 dp_2 \wedge \betat_1 \wedge d p_1 \wedge \betat_2 = 2 dp_2 \wedge dp_1 \wedge \betat_1 \wedge \betat_2  \\
    &= 2 dp_2 \wedge dp_1 \wedge (\beta_1 \wedge \beta_2 + d(f(\phi_1 + \phi_2)) \wedge (\beta_2 - \beta_1)) \\
    &= 2 dp_2 \wedge dp_1 \wedge \beta_1 \wedge \beta_2 + 2(\phi_1 + \phi_2) dp_2 \wedge dp_1 \wedge df \wedge (\beta_2 - \beta_1)
\end{align*}
We need to check positivity in two cases: inside and outside a neighborhood of radius $R_0$ of $B \subset \Sigma$.  Outside the neighborhood, then $\beta_1 \wedge \beta_2$ is strictly positive and bounded below, therefore the 4-form is positive provided that $|\phi_1 + \phi_2|$ is sufficiently small.  Inside this neighborhood, then $df = 0$ and the form is nonnegative and equal to 0 if and only if $\beta_1 = \beta_2 = 0$, which occurs precisely along $B \times D^2$.

We then compute that
\begin{align*}
    dt \wedge \mu_{\Sigma,\lambda} \wedge \omega_0 &= (t + 2) dt \wedge ds \wedge \betat_{\lambda} \wedge \betat_{\lambda+1} \\
    &= (t + 2) \omega_0 \wedge \omega_0 \\
    ds \wedge \mu_{\Sigma,\lambda} \wedge \omega_0 &= (s-1) dt \wedge ds \wedge \betat_{\lambda} \wedge \betat_{\lambda+1} \\
    &= (s - 1) \omega_0 \wedge \omega_0 \\
    \zeta_{\lambda} \wedge \mu_{\Sigma,\lambda} \wedge \omega_0 &= (t -s + 3) \omega_0 \wedge \omega_0 \\
    dt \wedge (\mu_{\Sigma,\lambda} - 3\betat_{\lambda}) \wedge \omega_0 &= (t - 1) dt \wedge ds \wedge \betat_{\lambda} \wedge \betat_{\lambda+1} \\
    &= (t - 1) \omega_0 \wedge \omega_0
\end{align*}
\end{proof}

Recall that a (positive) {\it confoliation} on a 3-manifold is a 2-plane field defined as the kernel of a 1-form $\alpha$ satisfying $\alpha \wedge d \alpha \geq 0$.  A negative confoliation is a 2-plane field satisfying the converse inequality everywhere.

\begin{corollary}
For $\lambda = 1,2,3$,
\begin{enumerate}
    \item The form $\mu_{\Sigma,\lambda}$ defines a (positive) confoliation on $\widehat{Y}_{\lambda}$,
    \item the form $\mu_{\Sigma,\lambda}$ defines a positive confoliation on $H_{\lambda}$ (with its coorientation given by $dt$),
    \item The form $\mu_{\Sigma,\lambda+1}$ defines a negative confoliation on $H_{\lambda}$ (with its coorientation given by $dt$)
\end{enumerate}
\end{corollary}

\begin{proof}
The form $\zeta_{\lambda}$ coorients $\widehat{Y}_{\lambda}$, so the from $\mu_{\lambda,\Sigma}$ defines a confoliation if and only if
\[\zeta_{\lambda} \wedge \mu_{\Sigma,\lambda} \wedge d \mu_{\Sigma,\lambda} \geq 0\]
which holds by the previous lemma since $\mu_{\Sigma,\lambda}$ is a primitive for $\omega_0$.
\end{proof}

Recall that we have chosen Morse coordinates $(x,y)$ on a neighborhood $U$ of $B = \beta^{-1}(0)$.  We can assume that $f|_U$ is constant.  Let $R = x^2 + y^2$ and choose some fixed values $0 < R_0 < R_1 < R_2$.  Choose a cutoff function $\phi_B$ satisfying
\[\phi_B = \begin{cases} 1 & \text{for }R \leq R_1 \\ 0 & \text{for }R \geq R_2 \end{cases}\]
and define the 1-form
\[\mu_B = \phi_Bf (xdy - ydx)\]

\begin{proposition}
\label{prop:sigma-symplectic}
For $\delta > 0$ sufficiently small, the form 
\[\omega_{\Sigma} = \omega_0 + \delta d \mu_B\]
is symplectic on $\Sigma \times D^2$.  Furthermore, for each $\lambda = 1,2,3$, we have that
\begin{align*}
    \zeta_{\lambda} \wedge (\mu_{\Sigma,\lambda} + \delta \mu_B) \wedge \omega_{\Sigma} & > 0 \\
    dt \wedge (\mu_{\Sigma,\lambda} + \delta \mu_B) \wedge \omega_{\Sigma} & > 0 \\
    dt \wedge (\mu_{\Sigma,\lambda+1} + \delta \mu_B) \wedge \omega_{\Sigma} &< 0 \\
\end{align*}
\end{proposition}

\begin{proof}
By Proposition \ref{prop:first-symplectic-sigma}, we know that the forms are strictly positive away from $B \times D^2$.  Let $dvol_{\Sigma}$ be an area form on $\Sigma$.  Then we have that
\[\omega_0 \wedge \omega_0 = h dp_2 \wedge dp_1 \wedge dvol_{\Sigma}\]
where $h$ is a function on $\Sigma$ that is nonnegative and vanishes precisely at $B$.  On the complement of the subset $\{R \leq R_0\}$, the function $h$ is strictly positive and bounded below.  Therefore, for $\delta > 0$ sufficiently small, the form $\omega_{\Sigma}$ remains symplectic on the complement of $\{R \leq R_0\}$.
Furthemore, we have that
\begin{align*}
    \omega_0 &= dp_2 \wedge (\beta_1 + f dp_1) - dp_1 \wedge (\beta_2 + f dp_2) \\
    &= dp_2 \wedge \beta_1 - dp_1 \wedge \beta_2 + 2f dp_2 \wedge dp_1 \\
    &= dt \wedge \beta_{\lambda} - ds \wedge \beta_{\lambda+1} + 2f dt \wedge ds \\
    \mu_{\Sigma,\lambda} &= (t+2) (\beta_{\lambda} + f ds) - (s-1) (\beta_2 + f (dt - ds)) \\
    \mu_{\Sigma,\lambda+1} &= (t-1) (\beta_{\lambda} + f ds) - (s-1) (\beta_2 + f (dt - ds)) \\ 
    d\mu_B &= f dx \wedge dy
\end{align*}
In addition, since $\mu_B$ is a 1-form on a surface, we have that
\[\mu_B \wedge d \mu_B = 0 \qquad d \mu_B \wedge d \mu_B = 0\]
Therefore, to establish the proposition, we need to check the following four inequalities:
\begin{align*}
    \omega_0 \wedge d \mu_B &> 0 \\
    \zeta_{\lambda} \wedge (\mu_B \wedge \omega_0 + \mu_{\Sigma,\lambda} \wedge d \mu_B) &> 0 \\
    dt \wedge (\mu_B \wedge \omega_0 + \mu_{\Sigma,\lambda} \wedge d \mu_B) &> 0 \\
    dt \wedge (\mu_B \wedge \omega_0 + \mu_{\Sigma,\lambda+1} \wedge d \mu_B) &< 0 \\
\end{align*}
After expanding, we have the following individual terms:
\begin{align*}
    \omega_0 \wedge d \mu_B &= 2f^2 dp_2 \wedge dp_1 \wedge dx \wedge dy \\
    dt \wedge \mu_B \wedge \omega_0 &= dt \wedge ds \wedge \mu_B \wedge \beta_{\lambda+1} \\
    -ds \wedge \mu_B \wedge \omega_0 &= dt \wedge ds \wedge \beta_{\lambda} \wedge \mu_B \\
    dt \wedge \mu_{\Sigma,\lambda} \wedge d \mu_B &= f^2(t + s + 1)dt \wedge ds \wedge dx \wedge dy \\
    -ds \wedge \mu_{\Sigma,\lambda} \wedge d \mu_B &= f^2(1-s) dt \wedge ds \wedge dx \wedge dy\\
    dt \wedge \mu_{\Sigma,\lambda+1} \wedge d \mu_B &= f^2(t + s - 2)dt \wedge ds \wedge dx \wedge dy
\end{align*}
All except for the second and third terms are strictly positive for $|s|,|t|$ small.  For these terms, note that we can find some quadratic forms $Q_1,Q_2$  such that
\begin{align*}
    \beta_{\lambda} \wedge \mu_b &= Q_1 dx \wedge dy \\
    \mu_B \wedge \beta_{\lambda+1} &= Q_2 dx \wedge dy
\end{align*}
Therefore, for any nonzero constant $A$, we have that
\begin{align*}
    &A \omega_0 \wedge \omega_0 + \delta_1 dt \wedge \mu_B \wedge \omega_0 - \delta_2 ds \wedge \mu_B \wedge \omega_0 \\
    &= (A(x^2 + y^2) + \delta_1 Q_1 + \delta_2 Q_2) dt \wedge ds \wedge dx \wedge dy
\end{align*}
In particular, for $\delta_1,\delta_2$ sufficiently small we have that $A(x^2 + y^2) + \delta_1 Q_1 + \delta_2 Q_2$ remains a definite quadratic form.  Consequently, the second and third terms do not affect any of the positivity calculations.
\end{proof}

\begin{corollary}
\label{cor:contact}
For $\lambda = 1,2,3$ and $\delta > 0$ sufficiently small,
\begin{enumerate}
    \item the form $\mu_{\Sigma,\lambda}$ defines a positive contact structure on $\widehat{Y}_{\lambda}$,
    \item the form $\mu_{\Sigma,\lambda}$ defines a positive contact structure on $H_{\lambda}$ (with its coorientation given by $dt$),
    \item the form $\mu_{\Sigma,\lambda+1}$ defines a negative contact structure on $H_{\lambda}$ (with its coorientation given by $dt$),
\end{enumerate}
\end{corollary}

\subsection{Extending over the handlebodies}

In the previous subsection, we constructed an exact symplectic form $\omega = d\mu_{\Sigma,\lambda} + \delta d\mu_B$ in a neighborhood of the central surface.  We now describe how to extend it as a symplectic form across the handlebody $H_{\lambda}$.  Furthermore we will show that the form $$\alpha_{\lambda} = \mu_{\Sigma,\lambda} + \delta \mu_B$$ extends as a positive contact form and $$\alpha_{\lambda+1} = \mu_{\Sigma,\lambda+1} + \delta \mu_B = \mu_{\Sigma,\lambda} - 3 \betat_{\lambda} + \delta \mu_B$$ extends as a negative contact form across $H_{\lambda}$.

We extend $\mu_{\Sigma,\lambda} + \delta \mu_B$ as follows:
\begin{enumerate}
    \item Recall that
\[\mu_{\Sigma,\lambda} = (t +2) \betat_{\lambda} + (-s + 1) \betat_{\lambda+1}\]
To extend $\omega_0 = d \mu_{\Sigma,\lambda}$ across $H_{\lambda}$, choose a smooth function $\Phi(s)$ satisfying
\[\Phi(s) = \begin{cases} 1 - s & \text{if } s \leq 2 \epsilon \\ 0 & \text{if } s \geq 4\epsilon \end{cases}\]
that furthermore satisfies $\Phi' < 0$ for $s \in [2\epsilon,3\epsilon]$ and $\Phi' \leq \Phi \leq 0$ for $s \in [3 \epsilon, 4\epsilon]$.  We can extend the primitive as
\[\mu_{\Sigma,\lambda} = (t + 2) \betat_{\lambda} + \Phi(s) \betat_{\lambda+1}\]
\item Recall that $B$ can be extended to a 1-complex $\cS$ and the Morse coordinates on a neighborhood of $B$ can be extended across a tubular neighborhood of the 1-complex $\cS$.  In these coordinates, the form $\mu_B = f \phi_B (x dy - ydx)$ extends across $\nu(\cS)$.
\end{enumerate}

Finally, in the previous section, we constructed an exact 2-form $d \mu_{H,\lambda}$ that orients the leaves of $\cF_{\lambda}$ for $s \geq 5\epsilon$. 

The main technical calculation of this section is the following result.

\begin{proposition}
\label{prop:main-extension}
Consider the form
\[\alpha_{\lambda} = \mu_{\Sigma,\lambda} + \frac{1}{C} \mu_{H,\lambda} + \delta \mu_B\]
For $C > 0$ sufficiently large and $\delta > 0$ sufficiently small, we have that
\begin{enumerate}
    \item $d \alpha_{\lambda}$ is a symplectic form on a neighborhood of $H_{\lambda}$, 
    \item $\alpha_{\lambda}$ is a positive contact form on $H_{\lambda}$, and
    \item $\alpha_{\lambda} - 3 \betat_{\lambda}$ is a negative contact form on $H_{\lambda}$
\end{enumerate}
\end{proposition}

\begin{proof}
To prove part (1), we need to check that
\[d \alpha_{\lambda} \wedge d \alpha_{\lambda} > 0\]
and to prove part (2), we need to check that 
\[ dt \wedge \alpha_{\lambda} \wedge d \alpha_{\lambda} > 0\]
and to prove part (3), we need to check that
\[ dt \wedge (\alpha_{\lambda} - 3 \beta_{\lambda}) \wedge d \alpha_{\lambda} < 0 \]

We will break up the calculation into cases according to intervals in the $s$-coordinate, which are determined by where the various cutoff functions are defined.  Furthermore, in each case, we will first assume that $\delta = 0$ and prove that each term is nonnegative everywhere and strictly positive for $R \geq R_0$.  Then we will show that by choosing $\delta > 0$ sufficiently small, we can achieve strict positivity for $R \leq R_0$ without compromising it for $R \geq R_0$.

For $R \leq R_0$, we have that $\phi_B = 1$ and $f = \pm 1$.  Thereforem
\[ \mu_B = f (xdy - ydx) \qquad d\mu_B = 2f dx \wedge dy\]
Moreover, since $\mu_B$ is a 1-form on a surface, we have
\[\mu_B \wedge \mu_B =0 \qquad \mu_B \wedge d \mu_B = 0\]
Therefore,
\begin{align*}
    (d \alpha_{\lambda} + \delta d \mu_B) \wedge (d \alpha_{\lambda} + \delta d \mu_b) &= d \alpha_{\lambda} \wedge d \alpha_{\lambda} + 2 \delta (d \alpha_{\lambda} \wedge d \mu_B) \\
    dt \wedge (\alpha_{\lambda} + \delta \mu_B) \wedge (d\alpha_{\lambda} + \delta d \mu_B) &= dt \wedge \alpha_{\lambda} \wedge d \alpha_{\lambda} + \\
    &+ \delta( dt \wedge \mu_B \wedge d \alpha_{\lambda} + dt \wedge \alpha_{\lambda} \wedge d \mu_B )\\
    dt \wedge (\alpha_{\lambda} - 3 \betat_{\lambda} + \delta \mu_B) \wedge (d \alpha_{\lambda} + \delta d \mu_B) &= dt \wedge (\alpha_{\lambda} - 3 \betat_{\lambda}) \wedge d \alpha_{\lambda}  \\
    &+ \delta( dt \wedge \mu_B \wedge d \alpha_{\lambda} + dt \wedge \alpha_{\lambda} \wedge d \mu_B) \\
    & - 3\delta dt \wedge \betat_{\lambda} \wedge d \mu_B
\end{align*}
Consequently, it suffices to check that for $R \leq R_0$:
\begin{align*}
    d\alpha_{\lambda} \wedge d\mu_B & > 0 \\
    dt \wedge (\alpha_{\lambda} + \delta \mu_B) \wedge d \alpha_{\lambda} & > 0  & dt \wedge (\alpha_{\lambda} - 3 \beta_{\lambda} + \delta \mu_B) \wedge d \alpha_{\lambda} & < 0\\
    dt \wedge \alpha_{\lambda} \wedge d \mu_B & > 0 & dt \wedge (\alpha_{\lambda} - 3 \beta_{\lambda}) \wedge d \mu_B & < 0
\end{align*}

\noindent {\bf Case 0: $s \in [0,\epsilon]$}.  This is a tubular neighborhood of $\Sigma$ and the computations from the previous subsection apply. \\

\noindent {\bf Case 1: $s \in [ \epsilon,2\epsilon]$}.  This is where the functions $\{\phi'_p\}$ are supported. Here, the primitive is
\[\alpha_{\lambda} = (t + 2)\betat_{\lambda} + (1 - s) \betat_{\lambda+1} + \frac{1}{C}\sum_p \phi_p(s) PD (\gamma_p)\]
and so
\[d \alpha_{\lambda} = dt \wedge \betat_{\lambda} - ds \wedge \betat_{\lambda+1} + \frac{1}{C}\sum_p \phi'_p ds \wedge PD(\gamma_p)\]

Therefore
\begin{align*}
    d \alpha_{\lambda} \wedge d \alpha_{\lambda} &= dt \wedge ds \wedge \betat_{\lambda} \wedge \left(\betat_{\lambda+1} - \frac{1}{C} \sum_p \phi'_p PD(\gamma_p) \right) \\
    &= dt \wedge ds \wedge \betat_{\lambda} \wedge \betat_{\lambda+1} \\
    &- \frac{1}{C} \sum_p \phi'_p dt \wedge ds \wedge \betat_{\lambda} \wedge  PD(\gamma_p)  \\
    dt \wedge \alpha_{\lambda} \wedge d \alpha_{\lambda} &= (t + 2)dt \wedge ds \wedge \betat_{\lambda} \wedge \betat_{\lambda+1} \\
    &- \frac{t + 2}{C} \sum_p \phi'_p dt \wedge ds \wedge \betat_{\lambda} \wedge PD(\gamma_p) \\
    &+ \frac{1}{C} \sum_p \left( - \phi_p - \phi'_p(1 - s)\right) dt \wedge ds \wedge \betat_{\lambda+1} \wedge PD(\gamma_p) \\
    &- \frac{1}{C^2} \sum_{p,p'} \left(\phi_p \phi'_{p'} \right) dt \wedge ds \wedge PD(\gamma_p) \wedge PD(\gamma_{p'}) \\
        \end{align*}
    \begin{align*}
dt \wedge (\alpha_{\lambda} - 3\betat_{\lambda} )\wedge d \alpha_{\lambda} &= (t -1)dt \wedge ds \wedge \betat_{\lambda} \wedge \betat_{\lambda+1} \\
    &- \frac{t -1}{C} \sum_p \phi'_p dt \wedge ds \wedge \betat_{\lambda} \wedge PD(\gamma_p) \\
    &+ \frac{1}{C} \sum_p \left( - \phi_p - \phi'_p(1 - s)\right) dt \wedge ds \wedge \betat_{\lambda+1} \wedge PD(\gamma_p) \\
    &- \frac{1}{C^2} \sum_{p,p'} \left(\phi_p \phi'_{p'} \right) dt \wedge ds \wedge PD(\gamma_p) \wedge PD(\gamma_{p'})
\end{align*}
Fix some volume form $dvol_{\Sigma}$.  Then
\[dt \wedge ds \wedge \betat_{\lambda} \wedge \betat_{\lambda+1} = h_{\lambda} dt \wedge ds \wedge dvol_{\Sigma}\]
for some nonnegative function $h_{\lambda}$ that vanishes precisely at $B$.  Note that $PD(\gamma_p)$ is supported outside of the cylinder $R \leq R_0$.  On the complement of this cylinder, there is a lower bound on $h_{\lambda}$ independent of $C$. Therefore, the positive 4-form $dt \wedge ds \wedge \betat_{\lambda} \wedge \betat_{\lambda+1}$ dominates the remaining terms for $C$ sufficiently small on the entire complement of $R \leq R_0$.  Moreover, it is the only nonzero term for $R \leq R_0$, where it is nonnegative for all $R$ and strictly positive for $R > 0$. 

Finally, since $\psi'_p$ is not supported in $R \leq R_0$, we can leave this computation to Case 2 below.\\

\noindent {\bf Case 2: $s \in [2 \epsilon,3 \epsilon]$}.  Here, the primitive is
\[\alpha_{\lambda} = (t + 2)\betat_{\lambda} + \Psi(s) \betat_{\lambda+1}\]
with exterior derivative
\[d \alpha_{\lambda} = dt \wedge \betat_{\lambda} + \Phi' ds \wedge \betat_{\lambda+1}\]
Therefore we have
\begin{align*}
    d \alpha_{\lambda} \wedge d \alpha_{\lambda} &= -\Phi' dt \wedge ds \wedge \betat_{\lambda} \wedge \betat_{\lambda+1} \\
    dt \wedge \alpha_{\lambda} \wedge d \alpha_{\lambda} &= -\Phi'(t + 2) dt \wedge ds \wedge \betat_{\lambda} \wedge \betat_{\lambda+1}\\
    dt \wedge (\alpha_{\lambda}-3 \betat_{\lambda}) \wedge d \alpha_{\lambda} &= -\Phi'(t -1) dt \wedge ds \wedge \betat_{\lambda} \wedge \betat_{\lambda+1}
\end{align*}
which satisfy the required inequalities since $\Phi' < 0$.

For $R \leq R_0$, since $df = 0$ for $R \leq R_0$, we have that
\begin{align*}
    \betat_{\lambda} &= \beta_{\lambda} + f ds \\
    \betat_{\lambda+1} &= \beta_{\lambda+1} + f(dt - ds)
\end{align*}
so therefore
\begin{align*}
    \alpha_{\lambda} &= (t + 2)(\beta_{\lambda} + f ds) + \Phi (\beta_{\lambda+1} + f(dt - ds)) \\
    d \alpha_{\lambda} &= dt \wedge \beta_{\lambda} - \Phi' ds \wedge \beta_{\lambda+1} + 2f dt \wedge ds
\end{align*}
As a result
\[d \alpha_{\lambda} \wedge d\mu_B = 2f^2 dt \wedge ds \wedge dx \wedge dy > 0\]
and
\begin{align*}
dt \wedge (\alpha_{\lambda} + \delta \mu_B) \wedge d\alpha_{\lambda} &= -\Phi'dt \wedge ds \wedge ((t + 2)\beta_{\lambda} + \delta \mu_B) \wedge \beta_{\lambda+1} \\
&= -\Phi'((t+2)R + \delta Q) dt \wedge ds \wedge dx \wedge dy \geq 0 \\
dt \wedge (\alpha_{\lambda} - 3 \beta_{\lambda} + \delta \mu_B) \wedge d\alpha_{\lambda} &= -\Phi'dt \wedge ds \wedge ((t -1)\beta_{\lambda} + \delta \mu_B) \wedge \beta_{\lambda+1} \\
&= -\Phi'((t-1)R + \delta Q) dt \wedge ds \wedge dx \wedge dy \leq 0
\end{align*}
where $\mu_B \wedge \beta_{\lambda+1} = Q dx \wedge dy$.

Finally
\begin{align*}
    dt \wedge \alpha_{\lambda} \wedge d \mu_B &= \left( (t +2 - \Phi)f^2 \right) dt \wedge ds \wedge dx \wedge dy > 0 \\
     dt \wedge (\alpha_{\lambda} - 3 \beta_{\lambda}) \wedge d \mu_B &= \left( (t -1 - \Phi)f^2 \right) dt \wedge ds \wedge dx \wedge dy < 0 \\  
\end{align*}

This suffices to ensure the required inequalities for $R \leq R_0$ in this case. \\

\noindent {\bf Case 3: $s \in [3 \epsilon, 4\epsilon]$}.  This is where $\Phi,\Phi'$ go to 0 and contains the support of $\psi'_q$ (but not $\phi'_p$).  Here the primitive is
\[\alpha_{\lambda} = (t + 2)\betat_{\lambda} + \Phi(s) \betat_{\lambda+1} + \frac{1}{C} \sum_p \phi_p PD(\gamma_p) - \frac{1}{C} \sum_q \psi_q PD(\gamma_q) \]
and
\[d \alpha_{\lambda} = dt \wedge \betat_{\lambda} + \Phi' ds \wedge \betat_{\lambda+1} - \frac{1}{C}\sum_q \psi'_q ds \wedge PD(\gamma_q)\]
since each $\phi_p$ is constant in this region.  

Therefore
\begin{align*}
    d \alpha_{\lambda} \wedge d \alpha_{\lambda} &= (-\Phi') dt \wedge ds \wedge \betat_{\lambda} \wedge \betat_{\lambda+1} + \frac{1}{C} \sum_q \psi'_q dt \wedge ds \wedge \betat_{\lambda} \wedge PD(\gamma_q)
\end{align*}
The first term is positive on the intersection of the supports of $\Phi'$ and $\betat_{\lambda} \wedge \betat_{\lambda+1}$.  Similarly the second term is positive on the intersection of the supports of $\psi'_q$ and $PD(\gamma_q)$, since $\phi'_q \geq 0$ and $\gamma_q$ is positively transverse to the foliation $\cF_{\lambda}$.  In particular, this is strictly positive when $s \in [3\epsilon,4\epsilon]$ and $R \geq R_0$.

Next,
\begin{align*}
    dt \wedge \alpha_{\lambda} \wedge d \alpha_{\lambda} &= -(t + 2) \Phi' dt \wedge ds \wedge \betat_{\lambda} \wedge \betat_{\lambda+1} \\
    &- \frac{t + 2}{C} \sum_q (- \psi'_q) dt \wedge ds \wedge \betat_{\lambda} \wedge PD(\gamma_q) \\
    &- \frac{1}{C} \sum_q (\Phi' \psi_q-\Phi \psi'_q) dt \wedge ds \wedge \betat_{\lambda+1} \wedge PD(\gamma_q) \\
    &+ \frac{1}{C} \sum_p ( -\Phi' \psi_p) dt \wedge ds \wedge \betat_{\lambda+1} \wedge PD(\gamma_p) \\
    &+ \frac{1}{C^2}\sum_{p,q} \phi_p \psi'_q dt \wedge ds \wedge PD(\gamma_p) \wedge PD(\gamma_q) \\
    &+ \frac{1}{C^2}\sum_{q,q'} \psi_{q'} \psi'_{q} dt \wedge ds \wedge PD(\gamma_{q'}) \wedge PD(\gamma_{q}) 
\end{align*}
By assumption, the supports of $PD(\gamma_p),PD(\gamma_q)$ are contained in the support of $\beta_{\lambda} \wedge \beta_{\lambda+1}$.  Therefore, we can write
\begin{align*}
    dt \wedge ds \wedge \betat_{\lambda+1} \wedge PD(\gamma_p) &= g_p dt \wedge ds \wedge \beta_{\lambda} \wedge \beta_{\lambda+1} \\
    dt \wedge ds \wedge \betat_{\lambda+1} \wedge PD(\gamma_q) &= g_q dt \wedge ds \wedge \beta_{\lambda} \wedge \beta_{\lambda+1}
\end{align*}
for some functions $g_p,g_q$ on $\Sigma \times [-5 \epsilon,0]$ and
\begin{align*}
    dt \wedge ds \wedge PD(\gamma_p) \wedge PD(\gamma_q) &= h_{p,q} dt \wedge ds \wedge \beta_{\lambda} \wedge PD(\gamma_q) \\
    dt \wedge ds \wedge PD(\gamma_{q}) \wedge PD(\gamma_{q'}) &= h_{q,q'} dt \wedge ds \wedge \beta_{\lambda} \wedge PD(\gamma_{q'})
\end{align*}
for some functions $h_{p,q}$ and $h_{q,q'}$ on $\Sigma \times [0,6 \epsilon]$.  Consequently, we can group the first, second and fourth terms as
\[\left(-\Phi'(t + 2) + \frac{1}{C}\sum_q g_q (\Phi' \psi_q - \Phi \psi'_q) + \frac{1}{C} \sum_p g_p (-\Phi' \psi_p)\right) dt \wedge ds \wedge \beta_{\lambda} \wedge \beta_{\lambda+1}\]
Using the relation $\Phi' \leq \Phi \leq 0$ and $\psi'_q \leq 0$, this is bounded below by
\[-\Phi' \left( (t + 2) - \frac{1}{C} \sum_q |g_q|(|\psi_q| + |\psi'_q|) + \frac{1}{C} \sum_p g_p \psi_p \right) dt \wedge ds \wedge \beta_{\lambda} \wedge \beta_{\lambda+1}\]
which is positive for $C$ sufficiently large and $s \leq 4 \epsilon$.  Furthermore, we can group the third, fifth and sixth terms as
\[-\frac{1}{C} \sum_q \psi'_q \left(t + 2 - \frac{1}{C} \sum_p \phi_p h_{p,q} - \frac{1}{C} \sum_{q'} \psi_{q'}h_{q,q'} \right) dt \wedge ds \wedge \beta_{\lambda} \wedge PD(\gamma_q)\]
which is positive for $C$ sufficiently large.

Next, since
\begin{align*}
    dt \wedge \betat_{\lambda} \wedge d \alpha_{\lambda} &= -\Phi' dt \wedge ds \wedge \betat_{\lambda} \wedge \betat_{\lambda+1} - \frac{1}{C}\sum_q \psi'_q dt \wedge ds \wedge \betat_{\lambda} \wedge PD(\gamma_q)
\end{align*}
we have that
\begin{align*}
    & dt \wedge (\alpha_{\lambda} - 3 \betat_{\lambda}) \wedge d \alpha_{\lambda}
\end{align*}
is bounded above by
\begin{align*}
    -\Phi' \left( (t -1) + \frac{1}{C} \sum_q |g_q|(|\psi_q| + |\psi'_q|) + \frac{1}{C} \sum_p g_p \psi_p \right) dt \wedge ds \wedge \beta_{\lambda} \wedge \beta_{\lambda+1} \\
    -\frac{1}{C} \sum_q \psi'_q \left(t -1 - \frac{1}{C} \sum_p \phi_p h_{p,q} - \frac{1}{C} \sum_{q'} \psi_{q'}h_{q,q'} \right) dt \wedge ds \wedge \beta_{\lambda} \wedge PD(\gamma_q)
\end{align*}
which for $C$ sufficiently large is nonnegative everywhere and strictly positive when $s \in [3\epsilon,4\epsilon]$ and $R \geq R_0$.

As in Case 1 above, since $\psi_q$ is not supported in $R \leq R_0$ and $\psi_p$ is not supported when $R \leq R_0$ and $s \in [3 \epsilon,4\epsilon]$, this computation reduces to Case 2 above.
\\

\noindent {\bf Case 4: $s \geq 4 \epsilon$}.  This is where $\Phi$ vanishes.  Here, the primitive is
\[\alpha_{\lambda} = (t + 2)\betat_{\lambda} + \frac{1}{C}\mu_H\]
and 
\[d \alpha_{\lambda} = dt \wedge \betat_{\lambda} + \frac{1}{C} d \mu_H \]
So
\[d \alpha_{\lambda} \wedge d \alpha_{\lambda} = \frac{1}{C} dt \wedge \betat_{\lambda} \wedge d \mu_H\]
which by construction is positive.
In addition, we have that
\begin{align*}
    dt \wedge \alpha_{\lambda} \wedge d \alpha_{\lambda} &= \frac{1}{C} \left( dt \wedge \left( (t + 2) \betat_{\lambda} + \frac{1}{C} \mu_H \right) \wedge d \mu_H \right)
\end{align*}
which for $C$ sufficiently large is positive multiple of a perturbation of
\[ (t + 2) dt \wedge \betat_{\lambda} \wedge d \mu_H > 0\]
Similarly, the form
\[dt \wedge (\alpha_{\lambda} - 3\betat_{\lambda}) \wedge \alpha_{\lambda} = \frac{1}{C} \left( dt \wedge \left( (t -1) \betat_{\lambda} + \frac{1}{C} \mu_H \right) \wedge d \mu_H \right)\]
is a positive multiple of a perturbation of
\[(t -1) dt \wedge \betat_{\lambda} \wedge d\mu_H < 0\]
for $C$ sufficiently large. \\

For $R \leq R_0$, we can check that
\begin{align*}
    d \alpha_{\lambda} \wedge d \mu_B &= dt \wedge \betat_{\lambda} \wedge dx \wedge dy  \\
    dt \wedge \alpha_{\lambda} \wedge d \mu_B &= (t + 2) dt \wedge \betat_{\lambda} \wedge dx \wedge dy + \frac{1}{C} dt \wedge \mu_H \wedge dx \wedge dy \\
    dt \wedge  (\alpha_{\lambda} + \mu_B) \wedge d \alpha_{\lambda} &= \frac{1}{C}\left( (t + 2) dt \wedge \betat_{\lambda} \wedge d \mu_H + dt \wedge \mu_B \wedge d \mu_H + \frac{1}{C} dt \wedge \mu_H \wedge d \mu_H \right)\\
    dt \wedge (\alpha_{\lambda}- 3 \betat_{\lambda}) \wedge d \mu_B &= (t -1) dt \wedge \betat_{\lambda} \wedge dx \wedge dy + \frac{1}{C} dt \wedge \mu_H \wedge dx \wedge dy \\
    dt \wedge  (\alpha_{\lambda} - 3 \beta_{\lambda} + \mu_B) \wedge d \alpha_{\lambda} &= \frac{1}{C}\left( (t -1) dt \wedge \betat_{\lambda} \wedge d \mu_H + dt \wedge \mu_B \wedge d \mu_H + \frac{1}{C} dt \wedge \mu_H \wedge d \mu_H \right)
\end{align*}
The first, second and fourth terms are positive for $C$ sufficiently large.  Moreover, the second dominates the third and the fourth dominates the fifth for $C$ sufficiently large.
\end{proof}

\section{4-dimensional sectors}

\subsection{Contact structures on the 3-manifold pieces}

At this point, we have constructed a symplectic form on $\nu(H_1 \cup H_2 \cup H_3)$.  We now want to construct an inward-pointing Liouville vector field $\rho_{\lambda}$ along each boundary component $\widehat{Y}_{\lambda}$.  It is equivalent to constructing a primitive $\alpha_{\lambda}$ for $\omega$ in a neighborhood of $\widehat{Y}_{\lambda}$ that restricts to a contact form.  Then this 1-form $\alpha_{\lambda}$ is $\omega$-dual to the required Liouville vector field.

We now define the following 1-form over $\widehat{Y}_{\lambda}$:
\begin{align*}
   \widehat{\alpha}_{\lambda} &:= \begin{cases} \alpha_{\lambda}& \text{over $\nu(H_{\lambda})$} \\ 
   \mu_{\Sigma,\lambda} + \delta \mu_B & \text{over $\nu(\Sigma)$} \\
   \alpha_{\lambda-1} - 3 \betat_{\lambda-1} & \text{over $\nu(H_{\lambda-1})$} \end{cases}
\end{align*}
and 2-form
\[\omega := \begin{cases} d \alpha_{\lambda} & \text{over $\nu(H_{\lambda})$} \\ \omega_0 + d \mu_B & \text{over $\nu(\Sigma)$} \end{cases}\]

Combining Corollary \ref{cor:contact} and Proposition \ref{prop:main-extension} yields

\begin{proposition}
\label{prop:spine-cap}
The 2-form $\omega$ is well-defined and symplectic on $\nu(H_1 \cup H_2 \cup H_3)$.  Moreover, the form $\widehat{\alpha}_{\lambda}$ is well-defined, is a primitive for $\omega$, and is a contact form along $\widehat{Y}_{\lambda}$.
\end{proposition}

\subsection{Capping off with symplectic fillings}
The last remaining piece is to fill in the three concave ends of $(\nu(H_1 \cup H_2 \cup H_3),\omega)$ with Weinstein 1-handlebodies to obtain a closed symplectic 4-manifold.

We fix a standard model $(W_k,\omega_k,\rho_k,\phi_k)$ of a 4-dimensional Weinstein 1-handlebody, where 
\begin{enumerate}
    \item $W_k$ has a handle decomposition consisting of one 0-handle and $k$ 1-handles for $k \geq 0$,
    \item $\omega_k$ is an exact symplectic form,
    \item $\phi_k$ is a proper, exhausting Morse function with exactly one index-0 critical point and exactly $k$ index-1 critical points, and
    \item $\rho_k$ is a complete Liouville vector field that is gradient-like for the function $\phi_k$.
\end{enumerate}
By translation, we can assume that all critical values of $\phi_k$ are negative, so that the level set $Y_k = \phi^{-1}_k(0)$ is homeomorphic to either $S^3$ (if $k = 0$) or $\#_k S^1 \times S^2$ (if $k \geq 1$). Let $\alpha_k = \omega_k(\rho_k,-)$ be the induced contact form and $\xi_k$ the corresponding contact structure. The symplectization $(\RR \times Y, d(e^t \alpha_k))$ embeds symplectically into $(W_k,\omega_k)$ by flowing along the Liouville vector field $\rho_k$.

\begin{proposition}
If the contact structure $(\widehat{Y}_{\lambda},\widehat{\xi}_{\lambda})$ induced by the contact form of Proposition \ref{prop:spine-cap} is tight, then it can be filled in by a symplectic 1-handlebody.
\end{proposition}

\begin{proof}
There is a unique tight contact structure on $S^3$ and on $\#_k S^1 \times S^2$ for each $k \geq 1$, which is contactomorphic to the standard model $(Y_k,\xi_k)$.  Choose an identification $(\widehat{Y}_{\lambda},\widehat{\xi}_{\lambda})$ with $(Y_k,\xi_k)$ for the appropriate $k \geq 0$.  There exists some positive function $f: \widehat{Y}_{\lambda} \rightarrow \RR$ such that
\[ \widehat{\alpha}_{\lambda} = f \alpha_k\]
Consequently, we can symplectically identify a neighborhood of $\widehat{Y}_{\lambda} \subset (X,\omega)$ with a neighborhood of the graph of $\log f$ in the symplectization $(\RR \times Y_k,d(e^t \alpha_k))$.  We can then smoothly identify the 4-dimensional sector $\widehat{Z}_{\lambda}$ of the trisection with the sublevel set of the graph of $\log f$ to obtain a symplectic structure extending the one we constructed on the spine.
\end{proof}

\begin{proposition}
An $n$-parameter family of contact forms $\alpha_{\lambda}$ determines a $n$-parameter family of symplectic fillings.
\end{proposition}

\begin{proof}
An $n$-parameter family of contact forms $\widehat{\alpha}_{\lambda,t}$ corresponds to a family of positive functions of the form
\[\widehat{\alpha}_{\lambda,t} = f_t \alpha_k\]
This determines a family of graphs in $W_k$ and therefore a family of diffeomorphisms of $\widehat{Z}_{\lambda}$ into $W_k$.  Pulling back, we obtain a family of symplectic forms on $\widehat{Z}_{\lambda}$.
\end{proof}

\begin{remark}
The theorem of Laudenbach and Poenaru implies that each 4-dimensional sector $Z_{\lambda}$ can be glued in uniquely up to diffeomorphism.  Similarly, the symplectic structure we construct on $Z_{\lambda}$ is essentially unique as there is a unique tight contact structure on $Y_k$, which has a unique Weinstein filling up to deformation \cite[Theorem 16.9(c)]{CE12}.
\end{remark}

\bibliography{references}
\bibliographystyle{alpha}

\end{document}